\documentclass[onefignum,onetabnum]{siamart171218}

\usepackage{lipsum}
\usepackage{amsfonts}
\usepackage{graphicx}
\usepackage{epstopdf}
\usepackage{algorithmic}
\ifpdf
  \DeclareGraphicsExtensions{.eps,.pdf,.png,.jpg}
\else
  \DeclareGraphicsExtensions{.eps}
\fi

\newsiamremark{remark}{Remark}
\newsiamremark{hypothesis}{Hypothesis}
\crefname{hypothesis}{Hypothesis}{Hypotheses}
\newsiamthm{claim}{Claim}

\headers{An Iterative Block Matrix Inversion Algorithm for SPD Matrices}{A. Paterson, J. Pestana, V. Dolean}

\title{An Iterative Block Matrix Inversion (IBMI) Algorithm for Symmetric Positive Definite Matrices with Applications to Covariance Matrices\thanks{Submitted to the editors DATE.
\funding{Ann Paterson was funded by a University of Strathclyde International Strategic Partner (ISP) Research Studentship and the National Manufacturing Institute Scotland.}}}

\author{\textbf{Ann Paterson} \thanks{Department of  Mathematics and Statistics, University of Strathclyde, 
(\email{ann.paterson.2017@uni.strath.ac.uk}, 
}
\and Jennifer Pestana \thanks{Department of  Mathematics and Statistics, University of Strathclyde, 
(\email{jennifer.pestana@strath.ac.uk}}
\and Victorita Dolean 
\thanks{Department of Mathematics and Computer Science, Eindhoven University of Technology, 
  (\email {v.dolean.maini@tue.nl}}}

\usepackage{amsopn}

\makeatletter
\newcommand*{\addFileDependency}[1]{
  \typeout{(#1)}
  \@addtofilelist{#1}
  \IfFileExists{#1}{}{\typeout{No file #1.}}
}
\makeatother

\newcommand*{\myexternaldocument}[1]{%
    \externaldocument{#1}%
    \addFileDependency{#1.tex}%
    \addFileDependency{#1.aux}%
}

\ifpdf
\hypersetup{
  pdftitle={An Iterative Block Matrix Inversion (IBMI) Algorithm for Symmetric Positive Definite Matrices with Applications to Covariance Matrices},
  pdfauthor={A. Paterson, J. Pestana,  V. Dolean}
}
\fi
\usepackage{comment}
\usepackage{subcaption}   
\usepackage{multirow}
\usepackage{setspace} 
\myexternaldocument{ex_supplement}

\theoremstyle{plain}

\newcommand{\sig}{\mathcal{H}} 
\usepackage{caption}
\usepackage{subcaption}

\let\b\mathbf

\newcommand{\I}{\mathcal{I}}
\newcommand{\Q}{\mathcal{A}}
\newcommand{\Perm}{\mathcal{P}}
\newcommand{\z}{\mathbf{z}}

\newcommand{\tsig}{\tilde{\mathcal{H}}} 

\usepackage{dsfont}
\usepackage{todonotes}

\usepackage{algorithmic} 
\usepackage{cleveref}
\Crefname{ALC@unique}{Line}{Lines} 

\usepackage{xcolor}
\usepackage{soul} 

\newcommand{\one}[1]{\textcolor{black}{#1}}
\newcommand{\two}[1]{\textcolor{black}{#1}}

\begin{document}
\maketitle
\begin{abstract}
Obtaining the inverse of a large symmetric positive definite matrix $\mathcal{A}\in\mathbb{R}^{p\times p}$ is a continual challenge across many mathematical disciplines. The computational complexity associated with direct methods can be prohibitively expensive, making it infeasible to compute the inverse. In this paper, we present a novel iterative algorithm (IBMI), which is designed to approximate the inverse of a large, dense, symmetric positive definite matrix. The matrix is first partitioned into blocks, and an iterative process using block matrix inversion is repeated until the matrix approximation reaches a satisfactory level of accuracy. We demonstrate that the two-block, non-overlapping approach converges for any positive definite matrix, while numerical results provide strong evidence that the multi-block, overlapping approach also converges for such matrices.

\end{abstract}


\section{Introduction} \label{sec:lit review}
Finding the inverse of a large, symmetric positive definite matrix is crucial in various fields such as Bayesian inference \cite{Rue_2023}, computational physics \cite{FIND.ALG}, and medical imaging \cite{medical_imaging}. The difficulty in obtaining the inverse of a symmetric positive definite matrix $\Q \in \mathbb{R}^{p \times p}$, where $\sig=\Q^{-1}$, lies in the computational expense of doing so. Direct inversion techniques, such as those based on Gaussian elimination, can require $\mathcal{O}(p^3)$ flops and have a $\mathcal{O}(p^2)$ storage cost \cite[\textsection3.11]{Duff_chpt3}, making it infeasible to calculate the direct inverse for larger matrices. 
%
One well-known method to invert a (dense) symmetric positive definite matrix is to use the Cholesky factorisation to decompose a matrix $\Q$ into the product of lower triangular matrices $\Q= \b{L} \b{L}^{\top}$. Then, $\Q^{-1}$ is obtained by first solving the $p$ linear systems $\b{L} \b{z}_i = \b{e}_i$, where $\b{e}_i$ is the $i$th unit vector, and then solving $\b{L}^\top\b{h}_i = \b{z}_i$, where $\b{h}_i$ is the $i$th column of $\sig = \Q^{-1}$. 
These three steps to obtain $\sig$ can be combined into one sweep as described in \cite{SPD_inversion}. 
\\ \\ 
Alternatively, $p$ linear systems could be solved using a method such as the preconditioned conjugate gradient (PCG) method, which can solve large symmetric positive definite linear systems of the form $\Q\b{x} = \b{b}$. For dense matrices that can be represented using a hierarchical low-rank format, with invertible diagonal blocks, it is also possible to approximate the entire inverse (see, e.g., \cite[\textsection  2.8]{bebendorf}). 
\\ \\
\two{For well-conditioned matrices, the Newton--Schultz iteration can be used to compute the inverse of a matrix. If $\Q \in \mathbb{R}^{p \times p}$ is an invertible matrix, then $\Q^{-1}$ is found using the following recurrence equation based on Newton's method \cite[\textsection 4.1]{Householder_1965},
\begin{equation} \label{eq:newton}
    \b{X}_{n+1} = \b{X}_{n} \left(2\b{I}- \Q\b{X}_{n} \right),
\end{equation}
where \cref{eq:newton} seeks to satisfy $\Q\b{X} = \b{I}$. The iteration \cref{eq:newton} will converge if $\| \b{I}-\Q\b{X}_{0}\| < 1$, where $\b{X}_{0}$ is the initial guess and $\| \cdot \|$ is an appropriate matrix norm \cite{Pan_newton_inverse}. For symmetric positive definite matrices, the Newton--Schultz method will converge if $\rho(\b{I}-\Q\b{X}_{0}) <1$. However, if $\Q$ is ill-conditioned or $\b{X}_0$ does not approximate $\Q^{-1}$ well, then the iterates may fail to converge to $\Q^{-1}$. }
\\ \\ 
There also exist several methods to obtain the full (or partial) inverse of a large \textit{sparse} symmetric positive definite matrix. In 1973, Takahashi et al.\ derived a method for sparse matrix inversion \cite{TAKAHASHI}, that was further analysed by Erisman and Tinney \cite{Erisman.Taka}. The starting point for the method is the observation that, given a symmetric, non-singular matrix $\Q \in \mathbb{R}^{p \times p}$ and its $\b{LDL}^{\top}$ factorisation $\Q = \b{LDL}^{\top}$, the inverse satisfies:
\begin{align} \label{eq:taka1}
    \sig &= \b{D}^{-1}\b{L}^{-1} + \left(\b{I}-  \b{L}^{\top} \right) \sig.
\end{align}
The key observation is that $\b{I}- \b{L}^{\top}$ involves only the upper triangular part of $\Q$. Thus, if we wish to compute elements $\sig_{ij}$, $i\le j$ in the upper triangular part of $\sig$ (which, since $\Q$ is symmetric, also computes elements $\sig_{ji}$ in the lower triangular part), we can work with triangular matrices only. 
This leads to the recursive formula: 
\begin{equation} \label{eq:taka2}
    h_{ij} = d_{ij}^{-1} - \sum_{k>i}^{n} l_{ki}h_{kj} \ \text{ for} \ i\leq j,
\end{equation}
for elements of $\sig=\Q^{-1}$. By ordering $\Q$ such that the desired elements of $\sig$ will occur in its lower-right corner, we can compute these desired elements with fewer computations. 
The computational cost also depends on the sparsity of $\b{L}$, since we may find that many $l_{ik}=0$. We note that Rue and Martino \cite{RUE20073177} generalise the Takahashi recurrences to enable them to compute the marginal variances for Gaussian Markov random fields (GMRFs) with additional constraints. 
\\ \\ 
\two{
Within the field of optimisation, the inverse of a symmetric positive definite Hessian is needed, particularly for second-order methods. In stochastic optimisation the inverse of such a matrix is used to determine the parameter $\theta$ in the following problem \cite{stochastic_opti}:
$\theta = \arg\min_{h \in \mathbb{R}^h} G(h), \ \text{where}  \ \ G = \mathbb{E}\left[g(X,h) \right].$ The convex function  $G : \mathbb{R}^d \rightarrow \mathbb{R}$ is defined in terms of the twice differentiable function $g$ and a random vector $X$. 
The exact inverse of the Hessian $\mathcal{H} =\nabla^2G$ is usually unavailable, and must be approximated. One method approximates $\mathcal{H}^{-1}$ using an algorithm based on the recursive Robbins--Monro method, then employs a truncation-based step to ensure positive definiteness.}
\\ \\
Obtaining the covariance matrix from its inverse, also known as the precision matrix, is a well-known challenge within multivariate statistics. The covariance matrix is a dense symmetric positive definite matrix $\sig \in \mathbb{R}^{p\times p}$, unlike its inverse, the precision matrix, $\Q \in \mathbb{R}^{p\times p}$ which is often sparse. If only the diagonal of $\sig$ is required, Hutchinson's stochastic estimator \cite{Hutchinson} can be applied: 
\begin{equation} \label{Hutch}
    \text{diag}(\sig) \approx  \left[ \sum_{k=1}^{K} \z_{k} \odot \mathcal{H} \ \z_{k}  \right] \oslash  \left[ \sum_{k=1}^{K} \z_{k} \odot \z_{k} \right],
\end{equation}
where elements of the random vectors $\z_{k}$, for  $k\in\{1,\ldots,K\}$, take the value $1$ or $-1$ with equal probability. Here, $\odot$ represents element-wise multiplication (the Hadamard product) of the vectors, and $\oslash$ represents their element-wise division. 
\\ \\ 
The full covariance matrix $\sig = \Q^{-1}  \in \mathbb{R}^{p\times p}$ can be approximated using a Monte Carlo method that first computes $N_s$ samples $\z_k \sim \mathcal{N}(0,\Q^{-1})$, $k = 1,\dotsc, N_s$ using, e.g., the approaches in \cite{ChowandSaad, pap_yuille_2010, papandyuille}. These samples are then used to form the Monte Carlo estimator mentioned in \cite{papandyuille}, which has the standard Monte Carlo convergence rate of $\mathcal{O}(N_s^{\frac{1}{2}})$:
\begin{equation} \label{MC Estimators}
    \b{\hat{\mathcal{H}}_{MC}} = \frac{1}{N_{s}}\sum_{j=1}^{N_{s}} \b{z}^{j}\b{z}^{j\top}  =\frac{1}{N_{s}}\b{ZZ}^\top, 
    \qquad 
    \b{Z} =[\z_1,\z_2, \ldots, \z_{N_s}].
\end{equation}
In 2018, Sidén et al. \cite{Finn} developed three Rao-Blackwellized Monte Carlo (RBMC) estimators for approximating elements of $\sig$ that improve on \eqref{MC Estimators} by combining it with the Law of Total Variance. 
One of these, the Block RBMC estimator,  approximates a principal sub-matrix of $\sig$. The block estimator requires two sets $\I$ and $\I^c$ that partition the row/column indices of $\Q \in \mathbb{R}^{p\times p}$, i.e., $ \I \cup \I^c = \{1, \ldots, p\}$, $\I\cap \I^c = \emptyset$. The matrix $\hat{\sig}_{\I}$ is then defined to be the principal sub-matrix of the approximate inverse, corresponding to the elements in the rows and columns indexed in the set $\I$. The block RBMC estimator is then defined as: 
\begin{align}
\hat{\sig}_{\I} &\approx \Q_{\I}^{-1}+\frac{1}{N_s} \Q_{\I}^{-1} \Q_{\I, \I^{c}} \b{Z}_{\I^{c}} \left(\b{Z}_{\I^c}\right)^{\top} \! \!  \left(\Q_{\I,\I^c}\right)^{\top} \Q_{\I}^{-1}. \label{eq:block RBMC} \vspace{-0.20cm}
\end{align}
As for the simple Monte Carlo estimator \eqref{MC Estimators}, here $N_s$ is the number of Gaussian samples $\z_k \sim \mathcal{N}(0,\Q^{-1})$, while $\b{Z}_{\I^{c}}$ represents the sub-matrix of $\b{Z}$ in \eqref{MC Estimators} formed from the rows indexed by $\I^{c}$. 
When $|\I|=1$, the Block RBMC estimator becomes the simple RBMC estimator described in \cite{Finn}, which can compute one marginal variance. The authors also describe an iterative interface method, based on the Block RBMC estimator in \cite{Finn}, that can more accurately approximate the diagonal of $\sig$ than Hutchison's estimator in \eqref{Hutch} but at a higher computational cost. The iterative interface method is designed to compute selected elements of the covariance matrix, but it cannot approximate all elements of $\sig$ simultaneously \cite[\textsection 3.2.2]{Finn}. 
\\ \\
Zhumekenov et al., \cite{Rue_2023} presented an alternative method of selected inversion for spatio-temporal Gaussian Markov random fields (GMRFs) which includes recovering the marginal variances starting from the precision matrix. Their method is a hybrid approach, taking inspiration from the RBMC estimators from Sidén et al. \cite{Finn}, and Krylov subspace methods, which are becoming increasingly popular for solving large linear systems in multivariate statistics. 

\subsection{Main Contributions}
The existing literature provides various methods for computing selected elements of the inverse of a symmetric positive definite matrix. However, there is still a notable gap of approaches which can accurately and efficiently approximate a \textbf{full} inverse, as current methods are not able to accurately approximate all the off-diagonal elements.  
In this paper, we introduce the following contributions, which aim to reduce this gap.
\begin{itemize}
    \item \textbf{Novel iterative block matrix inversion algorithm (IBMI).} We advance the current literature by proposing a novel block matrix inversion algorithm, designed to efficiently approximate the \emph{whole} inverse of a dense symmetric positive definite matrix. Using the Block RBMC estimator as a starting point, we establish a link between \cref{eq:block RBMC} and block matrix inversion. A breakdown of how the algorithm iteratively updates the approximated inverse through block matrix inversion will be provided. Notably, our algorithm achieves an accurate approximation of the inverse not only for the principal sub-matrices, but also for the off-diagonal elements, addressing a significant limitation with current methods. 
    \item \textbf{Analysed convergence, cost, and error bound.} When $\Q$ is partitioned into two non-intersecting sets, the algorithm is guaranteed to converge for any symmetric positive definite matrix $\Q$. This has been shown both theoretically and numerically, and a bound is derived for the error after each iteration. When the algorithm is generalised to the multi-block overlapping case, numerical results show that the algorithm can also converge. Additionally, 
    we show that the algorithm can outperform direct methods such as MATLAB's inversion function (\texttt{inv}). This advantage is further explored in the breakdown of the cost of the algorithm, where we show when the algorithm converges in one iteration it can outperform direct methods in terms of complexity.

    \item \textbf{Applications.} The algorithm is applicable to any symmetric positive definite matrix. However, we choose to focus on covariance matrices when performing numerical experiments. This was motivated by 
    applications that require the inverse of a covariance matrix, known as the precision matrix, in multivariate statistics and data science e.g., Gaussian process regression \cite[\textsection 2]{GPR_application}. Some of the literature reviewed in \cref{sec:lit review} focussed on the 
    inversion of sparse symmetric positive definite matrices. The IBMI algorithm is a novel method which can obtain the inverse of both sparse and \emph{dense} symmetric positive definite matrices.
\end{itemize}
The paper is structured as follows; \cref{sec:background} details the novel iterative block matrix inversion algorithm. The convergence of the IBMI algorithm is proven in \cref{sec:conv} and the computational cost is discussed in \cref{sec:cost}. Numerical results in \cref{sec:graphs} will confirm theoretical findings and illustrate the performance of the IBMI algorithm on cases not covered by the theory. Finally, a discussion will conclude the paper in \cref{sec:conclusions}.



\section{An Iterative Algorithm for Matrix Inversion} \label{sec:background}
%
The motivation for, and development of, the iterative block matrix inversion algorithm (IBMI) will be detailed in this section. We first start by making the link between the Block RBMC estimator in \cref{eq:block RBMC} and block matrix inversion. Details of the IBMI algorithm will then be given, first for the simplest partitioning -- the two-block, non-overlapping case -- and then for the multi-block overlapping case. 

\subsection{\two{Approximate Block Matrix Inversion}}
\two{The Block RMBC estimator can be viewed in terms of approximate block matrix inversion, which is also the foundation of our IBMI algorithm. Hence, in this section we introduce the (approximate) block matrix inversion formula, and elucidate its connection to the block RBMC approach.} 
\\ \\ 
\two{We begin by describing the exact block matrix inversion formula. To this end,} recall the two index sets, $\I$ and $\I^c$,  from  \cref{eq:block RBMC} that partition the row and column indices of $\Q \in \mathbb{R}^{p\times p}$, and that satisfy $\I \cup \I^c = \{1, \ldots, p\}$, $\I\cap \I^c = \emptyset$. We permute the matrix $\Q\in\mathbb{R}^{p\times p}$ so that the rows and columns corresponding to indices in $\I$ appear first, and then partition this permuted matrix $\Perm\Q\Perm^\top$ as: 

\begin{equation} \Perm\Q \Perm^\top= {\begin{bmatrix} 
 \Q_{\I} & \Q_{\I,\I^c} \\ 
 \Q_{\I^c,\I} & \Q_{\I^c}
 \end{bmatrix}} , \qquad \text{where} \quad \I \cup \I^c = \{1,\dotsc,p\}.
\end{equation}
The matrix $\Q_{\I}$ has rows and columns indexed by $\I$, $\Q_{\I^c}$ has rows and columns indexed by $\I^c$, $\Q_{\I,\I^c}$ has rows indexed by $\I$ and columns by $\I^c$ and $\Q_{\I^c,\I}=\Q_{\I,\I^c}^\top$. Then, the well known block matrix inversion formula (see, e.g., \cite[pg.19]{SchurComplementTxbk}) gives: 
\vspace{-0.15cm}
\begin{equation}
\label{BMIQ}
\begin{aligned}
\Perm \!\Q^{-1} \! \Perm^\top& = \! 
\begin{bmatrix}     
    \Q_{\I}^{-1} \! + \! \Q_{\I}^{-1}\Q_{\I,\I^c} \sig_{\I^c} \left(\Q_{\I,\I^c}\right)^{\top} \! \Q_{\I}^{-1} & 
    -\Q_{\I}^{-1}\Q_{\I,\I^c} \sig_{\I^c} \\ 
    - \sig_{\I^c}\! \left(\Q_{\I,\I^c}\right)^{\top} \!\Q_{\I}^{-1} & 
    \sig_{\I^c}
    \end{bmatrix} \\
     & = \! 
    \begin{bmatrix}
        \sig_{\I} & \! \! \! \! \sig_{\I,\I^c} \\ 
         \sig_{\I^c,\I} & \! \! \! \!\sig_{\I^c}
    \end{bmatrix} \ = \ \Perm\sig\Perm^\top,
\end{aligned}
\end{equation}
where $\sig_{\I^c} =  (\Q_{\I^c} - (\Q_{\I,\I^c})^{\top} \Q_{\I}^{-1} \Q_{\I,\I^c})^{-1}$ is the inverse of the Schur complement. Inverse permutations can then be applied to recover $\sig=\Q^{-1}$. A link can now be made with the Block RBMC estimator, as the top left principal sub-matrix in \cref{BMIQ} looks almost equal to the Block RBMC estimator \cref{eq:block RBMC}, which can be rewritten as:
\begin{align*}
    \hat{\sig}_{\I} \approx  \Q_{\I}^{-1} + \Q_{\I}^{-1}\Q_{\I,\I^c} \tsig_{\I^c}\Q_{\I^c,\I} \Q_{\I}^{-1} \approx \sig_{\I}, 
    \qquad 
     \tsig_{\I^c}= \frac{1}{N_s}&\b{Z}_{\I^c} 
     (\b{Z}_{\I^c})^{\top}.
\end{align*}
Thus, by approximating the inverse of the Schur complement $\tsig_{\I^c}$, an approximation of the top left principal sub-matrix $\tsig_{\I}$, can be obtained. Crucially, approximations to the off-diagonal sub-matrices of the first matrix in  \eqref{BMIQ} can also be obtained without additional computations (because  $\Q_{\I}^{-1} \Q_{\I,\I^c} \tsig_{\I^c}$ is required to compute $\tsig_{\I}$) and an approximation of the complete matrix $\tsig$ can be obtained. 
The resulting approximated matrix $\tsig$ is:
\begin{align} \label{eq:IBMI}
    \Perm\tsig\Perm^\top = 
    \begin{bmatrix}
        \Q_{\I}^{-1}+ \Q_{\I}^{-1} \Q_{\I,\I^c} \tsig_{\I^c} \Q_{\I^c,\I}\Q_{\I} ^{-1} & 
         -\Q_{\I}^{-1} \  \Q_{\I,\I^c}\tsig_{\I^c}\\ 
         - \tsig_{\I^c}\Q_{\I^c,\I}\Q_{\I}^{-1} &  \tsig_{\I^c}
    \end{bmatrix}.
\end{align}
The Monte Carlo estimator in \eqref{MC Estimators} could be used for the Schur complement approximation $ \tsig_{\I^c}$, but this is certainly not the only choice. Other possible choices for the initial guess will be discussed at the end of \cref{sec:alg}.

\subsection{ \two{ The Two-Block Non-Overlapping Case}}\label{sec:two-block}
Numerical evidence suggests the approximation in \cref{eq:IBMI} may not be very accurate, as $|\tsig_{ij} - \sig_{ij}|$ $i,j = 1,\dotsc p$, may be large when $|i-j|$ is large, i.e., elements in the off-diagonal blocks may be poorly approximated. To measure this initial approximation, symmetric positive definite matrices were generated using the RBF covariance kernel (given in \cref{tab:kernels}, discussed in \cref{sec:graphs}) and the error of the first approximation was recorded using the error estimate in \cref{eq:error}. The smallest matrix, of dimension $p=2^6$, had an error of 0.856886. As the dimension of the matrix increased, the error increased linearly, and the largest matrix, of dimension $p=2^{14}$, had an error of 20.9872. This trend was consistent with other matrices tested. \two{This could be analogous to the solution of linear systems involving discretised PDE, where it is known that the transfer of information for far away points is usually slow.} \\ \\
This initial approximation can be improved by iteratively updating the matrix, as we describe in this section. The key idea involves choosing different sets of indices for $\I$, and applying the block matrix inversion formula in \eqref{eq:IBMI} using elements of the most recently computed $\tsig$ to approximate $\tsig_{\I^c}$. 
\\ \\ 
%
For simplicity, the two-block non-overlapping case for a matrix $\Q \in \mathbb{R}^{p\times p}$ will be discussed here. In this case, two non-intersecting sets, $\I_1$ and $\I_2$, are introduced, where $ \I_1 \cup \I_2 = \{1,2,\ldots, p\},\  \I_1 \cap \I_2 = \emptyset$. At each iteration, denoted $r = 1,2,\dotsc$,  we cycle through these two sets, with the current set indicated by $k \in \{1,2\}$. The notation $\tsig^{(r,k)}$ is used to keep count of the iteration and set, $r$ and $k$, when updating the approximated matrix $\tsig$. Additionally, permutation matrices are denoted by $\Perm_k \in \mathbb{R}^{p\times p}$, where $\Perm_k$ permutes the rows of a matrix so that those indexed by elements of $\I_k$ appear before those indexed by elements of $\I_k^{c}$.
\\ \\ %
\textbf{Iteration 1, Set 1.}
We first set the iteration index $r=1$. Then, the set index $k=1$ is used to determine $\I$ in \eqref{eq:IBMI}, i.e., we let $\I=\I_1$ and $\I^c=\I_1^c$. An initial guess is made for the inverse of the Schur complement, $ \tsig^{(0,1)}_{\I_1^c}$, \two{ with the superscript representing the first guess $r=0$, using the first set $k=1$. This} is substituted into \eqref{eq:IBMI} to give the first approximation: 
\begin{equation} \label{eq:iter1 set1}
\begin{aligned}
    \Perm_1\tsig^{(1, 1) } _{\I_1}\Perm_1^{\top} &= \begin{bmatrix}\Q^{-1}_{\I_1} \! +\! \Q^{-1}_{\I_1}  \Q_{\I_1,\I_1^c}  \boxed{{}{\tsig}^{(0,1)}_{\I_1^c}} \  \Q_{\I_1^c,\I_1} \Q_{\I_1}^{-1} &
- \Q_{\I_1}^{-1} \Q_{\I_1,\I_1^c}  \boxed{{} \tsig^{(0,1)}_{\I_1^c}} \ \\
- \ \boxed{{}\tsig^{(0,1)}_{\I_1^c}} \ \Q_{\I_1^c,\I_1} \Q_{\I_1}^{-1} &  \boxed{{}\tsig^{(0,1)}_{\I_1^c}} 
\end{bmatrix}\\
&=\begin{bmatrix}
\tsig_{\I_1}^{(1,1)} & \tsig_{\I_1, \I_1^{c}}^{(1,1)} \\
\tsig_{\I_1^{c},\I_1}^{(1,1)} & \tsig_{ \I_1^{c}}^{(0,1)} 
\end{bmatrix}
. 
\end{aligned}
\vspace{0.15cm} 
\end{equation}
Note that having just two, non-overlapping sets leads to the special case where $\I_1^{c}=\I_2$, and vice versa.  Hence, $\Q_{\I_1} \equiv \Q_{\I_2^c}$ and $\Q_{\I_2} \equiv \Q_{\I_1^c}$. Therefore $\Perm_1\tsig^{(1, 1) }_{\I_1}\Perm_1^{\top}$, can be re-written as: \vspace{-0.3cm}
\begin{equation*}
\Perm_1\tsig^{(1, 1)}_{\I_1}\Perm_1^{\top} = 
    \begin{bmatrix}
\tsig_{\I_1}^{(1,1)} & \tsig_{\I_1, \I_2}^{(1,1)} \\
\tsig_{\I_2,\I_1}^{(1,1)} & \tsig_{\I_2}^{(0,1)} 
\end{bmatrix}.
\end{equation*}
%
%
%
\textbf{Iteration 1, Set 2.} 
Now, set $k=2$, so that $\I=\I_2$ and $\I^c=\I_2^c=\I_1$ in \eqref{eq:IBMI}. 
Then, an updated approximation of the matrix $\sig$ is obtained from \cref{eq:IBMI} and the permutation matrix $\Perm_2 \in \mathbb{R}^{p\times p}$. However, instead of using the initial guess, $\tsig_{\I_2}^{(0,1)}$, as an approximation of the inverse of the Schur complement, as in the previous approximation, we set $\tsig_{\I^c_2} = \tsig_{\I_1}^{(1,1)}$ in \cref{eq:IBMI}. 
The updated matrix approximation using $\I_2$ is then, 
\begin{equation} 
\label{eq:2block2}
\begin{aligned}
\Perm_2\tsig^{(1,2)}_{\I_2}\Perm_2^\top &= \begin{bmatrix}\Q^{-1}_{\I_2} \! +\! \Q^{-1}_{\I_2}  \Q_{\I_2,\I_2^c}  \boxed{{}{\tsig}^{(1,1)}_{\I_1}} \  \Q_{\I_2^c,\I_2} \Q_{\I_2}^{-1} &
- \Q_{\I_2}^{-1} \Q_{\I_2,\I_2^c}  \boxed{{} \tsig^{(1,1)}_{\I_1}} \ \\
- \ \boxed{{}\tsig^{(1,1)}_{\I_1}} \ \Q_{\I_2^c,\I_2} \Q_{\I_2}^{-1} &  \boxed{{}\tsig^{(1,1)}_{\I_1}} 
\end{bmatrix}\\
&=\begin{bmatrix}
\tsig_{\I_2}^{(1,2)} & \tsig_{\I_2, \I_1}^{(1,2)} \\
\tsig_{\I_1,\I_2}^{(1,2)} & \tsig_{ \I_1}^{(1,1)} 
\end{bmatrix}. 
\end{aligned}
\vspace{-0.1cm}  
\end{equation}
This completes one full iteration, as both sets have been used to update the approximate inverse $\tsig$. 
%
This iterative process then continues by incrementing $r$ and iterating through the index sets $k = 1,2$ as described above. In each case, the matrix $\tsig_{\I^c}$ in \eqref{eq:IBMI} is obtained from the most recently computed approximation of $\tsig$. 
For example, at the next step after \eqref{eq:2block2}, with $r=2$ and $k=1$,
the principal sub-matrix $\tsig_{\I_2}^{(1,2)}$ would be retained when calculating $\tsig^{(2,1)}$, as we would set $\tsig_{\I_1^c}^{(2,1)} = \tsig_{\I_2}^{(1,2)}$.\\\\
%
In general, 
\begin{equation} 
\label{eq:2blockk1}
\begin{aligned}
    \Perm_1\tsig^{(r, 1)} _{\I_1}\Perm_1^{\top} &= \begin{bmatrix}\Q^{-1}_{\I_1} \! +\! \Q^{-1}_{\I_1}  \Q_{\I_1,\I_1^c}  \tsig^{(r-1,1)}_{\I_1^c} \  \Q_{\I_1^c,\I_1} \Q_{\I_1}^{-1} &
- \Q_{\I_1}^{-1} \Q_{\I_1,\I_1^c}   \tsig^{(r-1,1)}_{\I_1^c} \ \\
- \tsig^{(r-1,1)}_{\I_1^c} \ \Q_{\I_1^c,\I_1} \Q_{\I_1}^{-1} &  \tsig^{(r-1,1)}_{\I_1^c} 
\end{bmatrix}\\
&=\begin{bmatrix}
\tsig_{\I_1}^{(r,1)} & \tsig_{\I_1, \I_1^{c}}^{(r,1)} \\
\tsig_{\I_1^{c},\I_1}^{(r,1)} & \tsig_{ \I_1^{c}}^{(r-1,1)} 
\end{bmatrix}
\end{aligned}
\vspace{0.15cm} 
\end{equation}
and 
\begin{equation} 
\label{eq:2blockk2}
\begin{aligned}
\Perm_2\tsig^{(r,2)}_{\I_2}\Perm_2^\top &= \begin{bmatrix}\Q^{-1}_{\I_2} \! +\! \Q^{-1}_{\I_2}  \Q_{\I_2,\I_2^c}  {\tsig}^{(r,1)}_{\I_1} \  \Q_{\I_2^c,\I_2} \Q_{\I_2}^{-1} &
- \Q_{\I_2}^{-1} \Q_{\I_2,\I_2^c}  \tsig^{(r,1)}_{\I_1} \ \\
- \tsig^{(r,1)}_{\I_1} \ \Q_{\I_2^c,\I_2} \Q_{\I_2}^{-1} &  \tsig^{(r,1)}_{\I_1} 
\end{bmatrix}\\
&=\begin{bmatrix}
\tsig_{\I_2}^{(r,2)} & \tsig_{\I_2, \I_1}^{(r,2)} \\
\tsig_{\I_1,\I_2}^{(r,2)} & \tsig_{ \I_1}^{(r,1)} 
\end{bmatrix}. 
\end{aligned}
\vspace{-0.1cm}  
\end{equation}
Before presenting the full novel iterative block matrix inversion algorithm, we first generalise the two-block, non-overlapping case to the multi-block overlapping case. Introducing multiple blocks is essential when handling large matrices, while overlap significantly improves the convergence rate by facilitating faster transfer of information between the blocks. 

\subsection{ \two{The Multi-Block, Overlapping Case}}
For larger matrices, $\Q_{\I_1}$ and $\Q_{\I_2}$ are too large to efficiently invert in \eqref{eq:IBMI}. The two-block case can be generalised to multiple blocks by partitioning the diagonal using multiple sets $\I_k$ for $k=1,\ldots, K$. 
\begin{equation*}
 \Q= \begin{bmatrix}
        \Q_{\I_{1}} & \cdot  & \cdot& \cdot \\ 
         \cdot &\Q_{\I_2} & \cdot& \cdot \\ 
         \cdot & \cdot & \ddots& \vdots \\ 
         \cdot & \cdot & \hdots &\Q_{\I_{K}}
    \end{bmatrix}. 
\end{equation*}
When using multiple sets, the blocks $\Q_{\I_k}$ are smaller and can be inverted much faster. At every iteration we cycle through $k=1,\dotsc,K$. For each value of $k$ we set $\I = \I_k$, and $\I^c = \{1,\dotsc, p\}\setminus \I_k$ in the approximate block matrix inversion formula \cref{eq:IBMI}, always using the most recently computed approximation to define $\tsig_{\I^c}$.
For example, if $K=4$ non-overlapping sets are used, we partition  $\Q$ 
as in \cref{eq: Multi 4 sets}. When $k=1$, we let $\I = \I_1$ and $\I^c = \I_2 \ \cup \  \I_3 \  \cup  \ \I_4$. A visual representation of this partitioning into the $2 \times 2$ structure, which is needed in \eqref{eq:IBMI} for block matrix inversion, is given by the right matrix in \cref{eq: Multi 4 sets} below, with dots representing off-diagonal block matrices. 
\begin{align}
\Q= \ 
\begin{array}{|c|c|c|c|} \hline \label{eq: Multi 4 sets}
        \Q_{\I_{1}} & \cdot & \cdot & \cdot \\ \hline
\cdot & \Q_{\I_{2}} & \cdot & \cdot \\ \hline
\cdot & \cdot & \Q_{\I_{3}} & \cdot \\ \hline
\cdot & \cdot & \cdot & \Q_{\I_{4}} \\ \hline 
\end{array}\ ,  \ \
\Perm_1\Q_{\I_1}\Perm_1^{\top}\!=\! \ 
\begin{array}{|c|ccc|}\hline
 \Q_{\I_{1}}& \phantom{\sig_{\I^c}^{(N)}} &   \cdot             &   \phantom{\sig_{\I^c}^{(N)}} \\ \hline
 \phantom{\sig_{\I^c}^{(N)}} &  & &\phantom{\sig_{\I^c}^{(N)}}  \\
   \cdot  & \phantom{\sig_{\I^c}^{(N)}} & \Q_{\I^c_1} &  \\
 &  &  \phantom{\sig_{\I^c}^{(N)}} &  \\ \hline
\end{array} \ . 
\end{align}
Next, we set $k=2$ and let  $\I=\I_2$ and $\I^c = \I_{1}  \ \cup \ \I_3 \  \cup \ \I_4$. We continue in this manner until all $K=4$ sets are used for $\I$ in \eqref{eq:IBMI} to complete the first iteration. 
\\\\
Overlap between the blocks is also introduced to speed up the convergence of the IBMI algorithm. The four-block partitioning with overlap is shown in \cref{eq: overlap_figure}. \one{Here, $\I=\I_1$ and $\I^c = \{ 1 , \ldots, p \}  \backslash \I_1$. The set $\I^c$ captures the elements in $\I_2$ that are not included in $\I_1$, (i.e., the elements of $\I_2$ that are not in the overlap) as well all elements in the remaining sets, namely $\I_3$ and $\I_4$.}
As for the non-overlapping case, a visual representation of the resulting partitioning into the $2 \times 2$ structure for \eqref{eq:IBMI} is given in \cref{eq: overlap with BMI}. A similar process is then repeated for the other three sets, $\I_2$, $\I_3$ and $\I_4$, to complete one iteration. 

\begin{align} 
\Q&= \ 
\begin{array}{|ccc ccc ccc|} \hline \label{eq: overlap_figure}
    \cdot & \cdot & \multicolumn{1}{c|}{\cdot} & \cdot & \cdot & \cdot & \cdot &\cdot &\cdot   \\ 
\cdot & \! \Q_{\I_{1}}  & \multicolumn{1}{c|}{\cdot} & \cdot & \cdot & \cdot & \cdot &\cdot &\cdot  \\ 
\cline{3-5}
\cdot & \multicolumn{1}{c|}{\cdot} & \multicolumn{1}{c|}{\cdot} & \cdot  & \multicolumn{1}{c|}{\cdot} & \cdot & \cdot &\cdot &\cdot  \\ 
\cline{1-3}
\cdot & \multicolumn{1}{c|}{\cdot}&\cdot & \Q_{\I_{2}}  &\multicolumn{1}{c|}{\cdot} &   \cdot & \cdot & \cdot &\cdot   \\ 
\cline{5-7}
\cdot & \multicolumn{1}{c|}{\cdot} & \cdot & \multicolumn{1}{c|}{\cdot}  & \multicolumn{1}{c|}{\cdot}& \cdot& \multicolumn{1}{c|}{\cdot} &\cdot &\cdot  \\ 
\cline{3-5}
\cdot & \cdot & \cdot &  \multicolumn{1}{c|}{\cdot} & \cdot & \Q_{\I_3}& \multicolumn{1}{c|}{\cdot}  &\cdot &\cdot  \\ 
\cline{7-9}
\cdot & \cdot & \cdot &  \multicolumn{1}{c|}{\cdot} & \cdot & \multicolumn{1}{c|}{\cdot}& \multicolumn{1}{c|}{\cdot}  &\cdot &\cdot  \\ 
\cline{5-7}
\cdot & \cdot & \cdot &  \cdot & \cdot & \multicolumn{1}{c|}{\cdot} &\cdot  & \Q_{\I_4}  &\cdot  \\ 
\cdot & \cdot & \cdot &  \cdot & \cdot& \multicolumn{1}{c|}{\cdot}& \cdot &\cdot &\multicolumn{1}{c|}{\phantom{a}\cdot \phantom{a}} \\ 
\hline
\end{array}\ ,  \\
\! \! \! \! \!\Perm_1\Q_{\I_1}\Perm_1^{\top} &=
\begin{array}{|ccc ccc ccc|} \hline \label{eq: overlap with BMI}
    \cdot & \cdot & \multicolumn{1}{c|}{\cdot} & \cdot & \cdot & \cdot & \cdot &\cdot &\cdot   \\ 
\cdot & \Q_{\I_{1}} & \multicolumn{1}{c|}{\cdot} & \cdot & \cdot & \cdot & \cdot &\cdot &\cdot  \\ 
\cdot & \cdot & \multicolumn{1}{c|}{\cdot} & \cdot  & \cdot & \cdot & \cdot &\cdot &\cdot  \\ \hline
\cdot & \cdot & \multicolumn{1}{c|}{\cdot}  & \cdot  &\cdot &   \cdot & \cdot& \cdot &\cdot   \\ 
\cdot & \cdot & \multicolumn{1}{c|}{\cdot}  & {\cdot}  & \cdot& \cdot& \cdot &\cdot &\cdot  \\ 
\cdot & \cdot & \multicolumn{1}{c|}{\cdot}  &  \cdot & \cdot & \Q_{\I_1^c} & \cdot &\cdot &\cdot  \\ 
\cdot & \cdot & \multicolumn{1}{c|}{\cdot} &  {\cdot} & \cdot & \cdot& \cdot  &\cdot &\cdot  \\ 
\cdot & \cdot & \multicolumn{1}{c|}{\cdot} &  \cdot & \cdot & \cdot &\cdot  & \cdot  &\cdot  \\ 
\phantom{a}\cdot \phantom{a}& \cdot & \multicolumn{1}{c|}{\phantom{a}\cdot \phantom{a}} &  \phantom{a}\cdot \phantom{a} & \cdot& \cdot & \cdot &\cdot  &\cdot   \\  \hline
\end{array} \ . 
\end{align} \vspace{-0.5cm}
\\ \\ %
%
In this section, we presented the IBMI algorithm for the particular case of two non-overlapping blocks. We then described the generalisation to the case of multiple, overlapping blocks. 
In the next section, we give the full IBMI algorithm for this general case, and discuss the choice of initial guess.

\subsection{Iterative Block Matrix Inversion (IBMI) Algorithm} \label{sec:alg}
The full iterative block matrix inversion algorithm is given in \cref{alg:IBMI}, which can be applied for $\I_k$ sets for $k=1, \ldots, K$. The algorithm will produce a final matrix $\tsig_{\text{final}}$ from the matrix $\Q$ and will also return the number of iterations $r$ taken to reach the desired tolerance level set by the user. 
\two{Within} \cref{alg:IBMI}\two{, the termination criterion uses the error estimate  \vspace{-0.3cm}
\begin{equation} \label{eq:error}
    \texttt{Error} = \left\| \tsig_{\I}\Q_{\I,\I^c} + \tsig_{\I,\I^c}\Q_{\I^c}  \right\|_{2}, 
\end{equation} 
where $\|\cdot\|_2$} \two{is the usual matrix norm induced by the Euclidean  norm. This measures the size of the upper-right, i.e., $(1,2)$, block of $\tsig\Q$, which is zero when $\tsig = \sig$. We note that alternative stopping conditions could be implemented. A tolerance is set by the user and if this error estimate is lower than the tolerance then \cref{alg:IBMI} will return the full approximated matrix and number of iterations.}
\\ \\
\two{The motivation for \cref{eq:error} comes from considering 
\begin{align*}
    \tsig\Q =& 
     \begin{bmatrix}
\Q_{\I}^{-1}+\Q_{\I}^{-1}\Q_{\I,\I^c}\tsig_{\I} \Q_{\I^c,\I}\Q_{\I}^{-1} &  -\Q_{\I}^{-1} \  \Q_{\I,\I^c}\tsig_{\I^c}\\ 
         - \tsig_{\I^c}\Q_{\I^c,\I}\Q_{\I}^{-1} &  \tsig_{\I^c}
    \end{bmatrix}
    \begin{bmatrix}
        \Q_{\I} & \Q_{\I,\I^c} \\ 
        \Q_{\I^c,\I} & \Q_{\I^c} 
    \end{bmatrix} \\ 
    = &
    \begin{bmatrix}
\mathbf{I}& \Q_{\I}^{-1}\Q_{\I,\I^c}\left[ \mathbf{I} -\tsig_{\I^c} \sig_{\I^c}^{-1} \right]\\ 
\mathbf{0} & -\tsig_{\I^c} \sig_{\I^c}^{-1}
\end{bmatrix},  
\end{align*} }
\two{where the exact Schur complement is denoted by $\sig_{\I^c}^{-1} = \Q_{\I^c,\I^c}- \Q_{\I^c,\I}\Q_{\I,\I}^{-1}\Q_{\I,\I^c}$. 
It is clear that when the approximation of the Schur complement $\tsig_{\I^c}$ is exact i.e., when $\tsig_{\I^c} = \sig_{\I^c}$, then $\tilde{\Q}^{-1} \Q = \b{I}$, as expected. Recall that the error estimate \cref{eq:error} is the upper off-diagonal block matrix of $\tsig\Q$.} 
\\ \\ 
\two{The decision to use \cref{eq:error} was due to the computational cost and the measure of how well approximated the off-diagonals blocks were. Compared to other error estimates where matrix inverses are needed, the two matrix-matrix products and one matrix-matrix addition used in \cref{eq:error} are less computationally expensive, even for dense matrices.  Additionally, as mentioned at the start of \cref{sec:two-block}, the elements in the off-diagonal blocks may be poorly approximated. Therefore, if the error is small in the off-diagonal blocks—where the approximation can be worse—it suggests that the approximation is at least as good, if not better, in the principal sub-matrices.} \\ \\ %
\two{We note that the framework of \cref{alg:IBMI} is similar to multiplicative Schwarz methods due to needing the inverse of principal sub-matrices to approximate $\sig$, and the multiplicative nature of the updates (see \cite{multi_Schwarz}).}
\begin{algorithm}[b!]
\caption{Iterative Block Matrix Inversion (IBMI) Algorithm} \label{alg:IBMI}\begin{algorithmic}
\STATE{Inputs: $\Q$, \texttt{tol, $\I_{k}$ for $k=1, \ldots , K$}, initial approximation $\tsig_{\I_1^c}^{(0,1)}$ of $\sig_{\I_1^{c}}$ in \cref{eq:IBMI}.}
\STATE{ \two{$r=0$}}
\STATE{\textbf{While} \texttt{error } $>$ \texttt{tol}}
\STATE{\two{$r=r+1$}}
\FOR{$k=1:K$}
\STATE{Determine $\I_k^c$.}
\IF {$k=1$} 
  \STATE{Get $\tsig_{\I_k^{c}}^{(r,k)}$ from $\tsig^{(r-1,K)}$.}
\ELSE
  \STATE{Get $\tsig_{\I_k^{c}}^{(r,k)}$ from $\tsig^{(r,k-1)}$.}
\ENDIF
\STATE{Use $\tsig_{\I_k^c}^{(r,k)}$ and $\Q$ in the block matrix inversion equation \eqref{eq:IBMI}}.
\STATE{Obtain updated approximation $\displaystyle \tsig^{(r,k)} = \begin{bmatrix}
        \tsig_{\I_k}^{(r,k)} & \tsig_{\I_k, \I_k^c}^{(r,k)} \\ 
        \tsig_{\I_k^c, \I_k}^{(r,k)} & \tsig_{ \I_k^c}^{(r,k)} \end{bmatrix}$.}
\ENDFOR
\STATE{Compute error estimate.}
\STATE{\textbf{Return:} $\tsig_{\text{final}} = \tilde{\sig}^{(r,K)}$ and number of iterations $r$.}
\end{algorithmic}
\end{algorithm}
 \two{ However, multiplicative Schwarz methods rely on just using the principal sub-matrices whereas \cref{alg:IBMI} uses information from $\Q$ and off-diagonal sub-matrices to find an approximation for $\tsig$. } 
 \\ \\
We end this section by remarking on the choice of the initial guess for \cref{alg:IBMI}. In our experiments, we take $\tsig_{\I_1^c}^{(0,1)}$ to be the identity matrix of the appropriate dimension. This initial guess still produces an accurate approximation $\tsig$ of $\sig=\Q^{-1}$ and we find that \cref{alg:IBMI} converges within a small number of iterations for our test matrices (see \cref{sec:graphs}). However, any symmetric positive definite approximation of $\tsig_{\I_1^c}$ can be used as an initial guess, \two{which is discussed further in \cref{sec:initial_guess}}. 

\section{Convergence of the IBMI algorithm}
\label{sec:conv}
In this section, the convergence of \cref{alg:IBMI} will be examined for the particular case of two non-overlapping blocks (cf.\ \cref{sec:two-block}). In this case, the diagonal blocks of the symmetric positive definite matrix $\Q$ are defined by the non-intersecting sets $\I_1$ and $\I_2$. 
Recall that in this case $\I_1^{c} = \I_2$ and $\I_2^{c} = \I_1$. 
The first step will be to show that the error at the $r$th iteration is related to the error in the initial guess.
\begin{lemma} \label{equ:lemma}
Let $\Q \in \mathbb{R}^{p \times p } $ be a symmetric positive definite matrix with inverse $\sig$, and let $\I_1$ and $\I_2$ be index sets such that $\I_1 \cup \I_2 = \{ 1,2, \ldots, p\}$, $\I_1 \cap \I_2 = \emptyset$. Let $\sig_{\I_2}$ be the sub-matrix formed from the rows and columns of $\sig$ indexed by $\I_2$, and let $\tsig_{\I_2}^{(r,2)}$ be the approximation of this matrix after $r$ complete iterations of \cref{alg:IBMI}. Then the error $\tsig^{\left(r,2\right)}_{\I_{2}} - {\sig}_{\I_{2}}$ at iteration $r$ satisfies, 
   \begin{equation}
       \label{lemma3.3}
   \tsig^{\left(r,2\right)}_{\I_{2}} - {\sig}_{\I_{2}}  \!=\left( \Q_{\I_2}^{-1} \Q_{\I_2,\I^c_2} \Q_{\I_1}^{-1} \Q_{\I_1,\I_2}\right)^{(r)} \left[\tsig_{\I_{2}}^{(0,2)} -  \sig_{\I_2}\right] \left( \Q_{\I_2,\I_1} \Q_{\I_1}^{-1} \Q_{\I_1,\I_2} \Q_{\I_2}^{-1} \right)^{(r)}.
   \end{equation}
\end{lemma}
\begin{proof}
To begin, we see from \cref{eq:2blockk1} that the upper diagonal block of the approximation,  $\tsig_{\I_1}^{(r,1)}$, at iteration $r$ is: 
\begin{align*} 
     \tsig_{\I_{1}}^{(r,1)} = \Q_{\I_1}^{-1} + \Q_{\I_1}^{-1} \Q_{\I_1,\I_2} \tsig_{\I_2}^{(r-1,2)} \Q_{\I_2,\I_1}\Q_{\I_1}^{-1}.
\end{align*}
This approximation is then used to update $\tsig_{\I_2} ^{(r,2)}$ using \cref{eq:2blockk2} to give 
\begin{equation}
\label{eq:approx_schur_recurrence}
\begin{aligned}
    \tsig_{\I_{2}}^{(r,2)} &= \Q_{\I_2}^{-1} + \Q_{\I_2}^{-1} \Q_{\I_2,\I_1} \tsig_{\I_1}^{(r,1)} \Q_{\I_1,\I_2} \Q_{\I_2}^{-1} \\ 
      &= \Q_{\I_2}^{-1} + \Q_{\I_2}^{-1} \Q_{\I_2,\I_1} \left[ \Q_{\I_1}^{-1} + \Q_{\I_1}^{-1} \Q_{\I_1,\I_2} \tsig_{\I_{2}}^{(r-1,2)}  \Q_{\I_2,\I_1}\Q_{\I_1}^{-1}   \right] \Q_{\I_1,\I_2} \Q_{\I_2}^{-1}. \\ 
\end{aligned}
\end{equation}
The exact Schur complement satisfies the same recurrence, since in this case \eqref{eq:IBMI} reduces to \eqref{BMIQ}. Hence, 
\begin{equation}
\label{eq:exact_schur_recurrence}
\begin{aligned}
    \sig_{\I_{2}}  &= \Q_{\I_2}^{-1} + \Q_{\I_2}^{-1} \Q_{\I_2,\I_1} \sig_{\I_1} \Q_{\I_1,\I_2} \Q_{\I_2}^{-1} \\     
    &= \Q_{\I_2}^{-1} + \Q_{\I_2}^{-1} \Q_{\I_2,\I_1} \left[ \Q_{\I_1}^{-1} + \Q_{\I_1}^{-1} \Q_{\I_1,\I_2} \sig_{\I_{2}}  \Q_{\I_2,\I_1}\Q_{\I_1}^{-1}   \right] \Q_{\I_1,\I_2} \Q_{\I_2}^{-1}, \\ 
\end{aligned}
\end{equation}
 where $ \sig_{\I_2} = \left( \Q_{\I^c} - \Q_{\I^c,\I} \Q_{\I}^{-1} \Q_{\I,\I^c}\right)^{-1}$.
It then follows from \cref{eq:approx_schur_recurrence} and \cref{eq:exact_schur_recurrence} that  
\begin{align*}
    \tsig^{\left(r,2\right)}_{\I_{2}} - {\sig}_{\I_{2}}  
    & = \Q_{\I_2}^{-1} \Q_{\I_2,\I_1} \Q_{\I_1}^{-1} \Q_{\I_1,\I_2} \left[\tsig_{\I_{2}}^{(r-1)} - \sig_{\I_2}\right]  \Q_{\I_2,\I_1}\Q_{\I_1}^{-1}  \Q_{\I_1,\I_2} \Q_{\I_2}^{-1}.
\end{align*}
By induction, on the iteration $r$, we
obtain the result.
\end{proof}
\cref{equ:lemma} can now be used to bound the error in the iterative block matrix inversion algorithm (\cref{alg:IBMI}), as we now show. \\ %
\begin{theorem}\label{lemma}
Let $\Q \in \mathbb{R}^{p\times p}$ be a symmetric positive definite matrix with inverse $\sig$, and let $\I_1$ and $\I_2$ be index sets such that $\I_1 \cup \I_2 = \{1,2,\dotsc,p\}, \ \I_1 \cap \I_2 = \emptyset$. Let $\sig_{\I_2}$ be the sub-matrix formed from the rows and columns of $\sig$ indexed by $\I_2$, and let $\tsig_{\I_2}^{(r,2)}$ be the approximation of this matrix after $r$ complete iterations of \cref{alg:IBMI} with sets $\I_1$ and $\I_2$.  
Then, the error in $\tsig_{\I_2}^{(r,2)}$ can be bounded by, 
\begin{equation}
\label{eq:bound}
\left\|{\tsig}_{\I_{2}}^{(r,2)} - {\sig}_{\I_{2}} \right\|_2 \leq \left\| \left( \Q_{\I_2}^{-1} \Q_{\I_2,\I_1} \Q_{\I_1}^{-1} \Q_{\I_1,\I_2}\right)\right\|_{2}^{2r} \left\| \tsig^{(0,2)}_{\I_2} -  \sig_{\I_2} \right\|_{2}. 
\end{equation}
Moreover, the iterative method will converge for any symmetric positive definite initial guess $\tsig^{(0,2)}_{\I_2}$. 
\end{theorem} 
\begin{proof}
Our goal will be to bound the norm of  $\sig_{\I_2}-\tsig_{\I_2}^{(r,2)}$, i.e., the error, after $r$ iterations, of the approximation to $\sig_{\I_2}$. 
Taking 2-norms of \eqref{lemma3.3} shows that 
\begin{align*}
    & \! \! \! \left\|{\tsig}^{(r,2)}_{\I_{2}} \!\! \! - {\sig}_{\I_{2}} \right\|_2 \! \! = \!  \left\| \left( \Q_{\I_2}^{-1}\! \Q_{\I_2,\I_1} \Q_{\I_1}^{-1} \Q_{\I_1,\I_2} \right)^{\!(r)} \! \! \left[ \tsig_{\I_2}^{(0,2)}\! \!\! \!-\! \sig_{\I_2} \! \right] \! \! \left( \Q_{\I_2,\I_1} \Q_{\I_1}^{-1} \Q_{\I_1,\I_2} \Q_{\I_2}^{-1} \! \right)^{\!(r)}\right\|_2 \\
&  \qquad =  \left\|\left( \Q_{\I_2}^{-1} \Q_{\I_2,\I_1} \Q_{\I_1}^{-1} \Q_{\I_1,\I_2} \right)^{r} \right\|_{2}^{2}\left\| \tsig_{\I_2}^{(0,2)}- \sig_{\I_2}\right\|_{2} \\ 
& \qquad \leq  \left\|\left( \Q_{\I_2}^{-1} \Q_{\I_2,\I_1} \Q_{\I_1}^{-1} \Q_{\I_1,\I_2} \right) \right\|_{2}^{2r} \left\|\tsig_{\I_2}^{(0,2)}- \sig_{\I_2}\right\|_{2}. 
\end{align*}
This proves the first part. \\ \\
The second part  follows from Theorem 7.7.7 in \cite[pg.497]{Horn_Johnson_2012} 
which shows that, whenever $\Q$ is symmetric positive definite, $\rho( \Q_{\I_1,\I_2}\Q_{\I_2}^{-1} \Q_{\I_2,\I_1} \Q_{\I_1}^{-1}) < 1$ where $\rho(\cdot)$ is the spectral radius. It then follows, by similarity, that 
that $\rho( \Q_{\I_2}^{-1} \Q_{\I_2,\I_1} \Q_{\I_1}^{-1}\Q_{\I_1,\I_2})<1$.
Finally, since $A^k \to 0$ as $k \to \infty$ if and only if $\rho(A)< 1$ for any square matrix $A$, we see that 
$\left\|\left( \Q_{\I_2}^{-1} \Q_{\I_2,\I_1} \Q_{\I_1}^{-1} \Q_{\I_1,\I_2} \right) \right\|_{2}^{2r} \rightarrow 0$ as $r \to \infty$.  
\end{proof}%

\begin{remark}
Although the current convergence analysis is limited to the two-block non-overlapping case, numerical experiments have suggested that the algorithm converges for any symmetric positive definite matrix when $K>2$ overlapping blocks are used. More evidence of this is detailed in \cref{sec:graphs}. 
\end{remark}%
\begin{corollary}
\one{Let $\Q$, $\I_1$ and $\I_2$ be as in \cref{lemma}. Let $\sig = \Q^{-1}$ and let $\tsig^{(r,2)}$ be the approximation of this matrix after $r$ complete iterations of \cref{alg:IBMI} with sets $\I_1$ and $\I_2$.  Then $\tsig^{(r,2)}$ converges to $\sig$ for any symmetric positive definite initial guess $\tsig_{\I_2}^{(0,2)}$. }
\end{corollary}%
\begin{proof}
\one{Let $\sig_{\I_2}$ and $\tsig_{\I_2}^{(r,2)}$ be as in \cref{lemma}. 
Then, we know from  \cref{lemma} that 
$\tsig_{\I_2}^{(r,2)}$ converges to $\sig_{\I_2}$. 
The next step of \cref{alg:IBMI} will, therefore, use the exact Schur complement in \cref{eq:IBMI}, which will then return the exact inverse $\sig$. Therefore, the two block non-overlapping case in \cref{alg:IBMI} will converge to the exact inverse of the whole matrix $\Q$.}
\end{proof}

\section{\one{Computational Cost of the IBMI Algorithm}} \label{sec:cost}
\one{The computational cost of one iteration of \cref{alg:IBMI} will now be discussed for the multi-block partitioning with overlapping blocks. (The non-overlapping case is recovered by setting the overlap to 0.) 
The most expensive operations of the algorithm are inverting the principal sub-matrices $\Q_{\I_k}$, $k = 1,\dotsc, K$, and performing matrix-matrix multiplications. Although the exact cost of these operations will depend on the properties of $\Q$, and the sets $\I_k$, the following analysis provides a sense of the cost per iteration. The flop\footnote{{Here, flop stands for floating point operations.}} counts for matrix-matrix and matrix-vector products are calculated according to \cite[pg.18]{geneMatrixComps}. For matrices $A \! \in \mathbb{R}^{q\times s}$, $B\in \mathbb{R}^{s\times t}$ and $C\in\mathbb{R}^{q\times t}$, and a vector $\b{v} \in \mathbb{R}^s$, the cost of computing $AB+C$ is $\mathcal{O}(2qst)$ flops, and the cost of computing $A\b{v}$ is $\mathcal{O}(2qs)$ flops. }
\\ \\ 
\one{Assume for simplicity that $\Q\in\mathbb{R}^{Km\times Km}$ is a dense, symmetric positive definite matrix, which is partitioned in $K$ overlapping sets $\I_k$ for $k=1,\ldots, K$. 
Care is needed when calculating the cost for the multi-block partitioning, as not every set $\I_k$ has the same amount of overlap; see \cref{eq: overlap with BMI}. The sub-matrices $\Q_{\I_1}$ and $\Q_{\I_K}$ will have half the number of overlapping elements compared to the other sub-matrices $\Q_{\I_k}$ for $k=2,\ldots, K-1$. Therefore, we first calculate the cost for $k=1$ and $k=K$, then consider $k = 2,\dotsc, K-1$ to  find an overall cost for one iteration of \cref{alg:IBMI}. We stress that for each set $\I_k$, we must calculate the cost of each sub-matrix of \cref{eq:IBMI},  
noting that the matrix-matrix product $\Q_{\I}^{-1}\Q_{\I,\I^c}$, or its transpose, appears four times. }
\subsection{\one{Cost of IBMI Algorithm for $k=1$ or $k=K$}}
\one{When $k = 1$ or $k = K$, the index set 
$\I_k$ is chosen such that $|\I_k| = m+h$, where $h$ is a fixed number of elements in the overlap. Additionally $| \I^c| = (K-1)m -h$. We break down the cost of calculating $\tsig_{\I}$ in \cref{eq:IBMI}, which will include calculating the cost of the off diagonal block $\tsig_{\I_k,\I_k^c}$ and the principal sub-matrix $\tsig_{\I_k}$. The most costly operation is computing $\Q_{\I}^{-1}$, which involves $\mathcal{O}((m+h)^3/3)$ flops to obtain a Cholesky factorisation and $\mathcal{O}(2(m+h)^3)$ flops to solve the linear systems required to find the inverse. The remaining operations to compute \cref{eq:IBMI} are matrix-matrix products and sums. Therefore, 
\begin{equation*}.
\begin{aligned} 
\text{Cost}\left(\Q^{-1}_{\I}\right) &=   \mathcal{O}\left(\frac{1}{3}(m+h)^3 + 2\left(m+h\right)^3 \right);  \\ 
\text{Cost}\left(\Q^{-1}_{\I}\Q_{\I,\I^c} \right)&=  \mathcal{O}\left(2(m+h)^2\left(\left(K-1\right)m-h\right) \right); \\ 
\text{Cost}\left(\Q_{\I}^{-1} \Q_{\I,\I^c} \tsig_{\I^c}  \right)&= \mathcal{O}\left(2(m+h)\left(\left(K-1\right)m-h\right)^2 \right); \\ 
\text{Cost}\left( \Q_{\I}^{-1} + \Q_{\I}^{-1} \Q_{\I,\I^c} \tsig_{\I^c} \Q_{\I^c,\I} \Q_{\I}^{-1} \right)&=   \mathcal{O}\left(2(m+h)^2\left(\left(K-1\right)m-h\right) \right). 
\end{aligned}%
\end{equation*} }
\one{Therefore, the cost of one application of \cref{eq:IBMI} for the set $\I_k$ where $k=1 \ \text{or} \ K$ is: 
\begin{equation*}
\begin{aligned}
    \text{Cost}\left(\tsig^{(r,k)}\right) &= \mathcal{O}\left(\frac{7}{3}(m+h)^3  \right) + \mathcal{O}\left(4(m+h)^2\left(\left(K-1\right)m-h\right) \right) + \\ 
    &  
    \qquad \mathcal{O}\left(2(m+h)\left(\left(K-1\right)m-h\right)^2 \right) \\  
     & = \mathcal{O}\left(\frac{1}{3}(m+h)^3 + 2K^2m^2\left(m+h\right)   \right).
\end{aligned}%
\end{equation*}} 
\vspace{-0.75cm}
\subsection{\one{Cost of IBMI Algorithm for $k= 2, \ldots, K-1$}}
\one{A similar process can be used to find the cost for the remaining sub-matrices. The sets $\I_k$ for $k=2,\ldots,K-1$ are chosen such that $|\I_k| = m+2h$ and $|\I^c | = (K-1)m-2h$. Therefore, the cost of one application of \cref{eq:IBMI} can be broken down as follows:
\begin{equation*}
\begin{aligned}
\text{Cost}\left(\Q^{-1}_{\I}\right) &=  \mathcal{O}\left(\frac{1}{3}(m+2h)^3 + 2\left(m+2h\right)^3 \right).  \\ 
\text{Cost}\left(\Q^{-1}_{\I}\Q_{\I,\I^c} \right)&=  \mathcal{O}\left(2(m+2h)^2\left(\left(K-1\right)m-2h\right) \right). \\ 
\text{Cost}\left(\Q_{\I}^{-1} \Q_{\I,\I^c} \tsig_{\I^c}  \right)&= \mathcal{O}\left(2(m+2h)\left(\left(K-1\right)m-h\right)^2 \right). \\ 
\text{Cost}\left( \Q_{\I}^{-1} + \Q_{\I}^{-1} \Q_{\I,\I^c} \tsig_{\I^c} \Q_{\I^c,\I} \Q_{\I}^{-1} \right)&=   \mathcal{O}\left(2(m+2h)^2\left(\left(K-1\right)m-2h\right) \right). 
\end{aligned}%
\end{equation*}%
The final cost for one application of \cref{eq:IBMI} in this case is: 
 \begin{align*}
    \text{Cost}\left(\tsig^{(r,k)}\right) &=  \mathcal{O}\left(\frac{7}{3}(m+2h)^3  \right) + \mathcal{O}\left(4(m+2h)^2\left(\left(K-1\right)m-2h\right) \right) + \\ 
    & \qquad \mathcal{O}\left(2(m+2h)\left(\left(K-1\right)m-2h\right) \right)^2 \\ 
   & = \mathcal{O}\left(\frac{1}{3}(m+2h)^3 + 2K^2m^2\left(m+2h\right)   \right).
\end{align*}}
\one{The total cost for one iteration of \cref{alg:IBMI} is therefore:
\vspace{-0.2cm}
\begin{align*}%
    \text{Cost} \left( \tsig^{(r,k)}\right) 
    &= \mathcal{O}\left(\frac{2}{3}(m+h)^3 + 4K^2m^2 \left(m+h\right)   \right)+ \\ 
    &
    \qquad \mathcal{O}\left((K-2)\left(\frac{1}{3}(m+2h)^3 + 2K^2m^2\left(m+2h\right) \right)   \right) \\ 
    &=\mathcal{O} \bigg( \frac{2}{3}(m+h)^3 + \frac{(K-2)}{3}(m+2h)^3  + 4K^2\left(K-1\right)m^2h  +2K^3m^3\bigg).
\end{align*} }
\one{When $\Q$ is partitioned into multiple sub-matrices with \textbf{no overlap}, the above equation reduces to: 
\begin{align*}
   \text{Cost} \left( \tsig^{(r,k)}\right)  
        &=  \mathcal{O}\left( \left( \frac{1}{3}+2K^2  \right)Km^3 \right). 
\end{align*}}
The initial guess of the inverse of the Schur complement can also be considered here, but since we use the identity matrix there is no additional cost. 
\\ \\ 
When $\Q$ is partitioned according to the multi-block \textbf{non-overlapping} case, \cref{alg:IBMI} can take many iterations to converge (see \cref{sec:graphs}). However, as we will see in \cref{sec:overlap} when a small amount of overlap is added between the diagonal blocks, \cref{alg:IBMI} can take just one iteration to converge. For these cases, the cost of \cref{alg:IBMI} can be compared with the cost of a direct solver. The cost of inverting $\Q \in \mathbb{R}^{Km \times Km}$ using the Cholesky factorisation and solving $Km$ linear systems would be $\mathcal{O}(\frac{1}{3}\left(Km\right)^3)+\mathcal{O}(2\left(Km\right)^3) = \mathcal{O}\left(\frac{7}{3} (Km)^3\right)$. Comparing this cost with the leading order term for the IBMI algorithm $\mathcal{O}\left( (\frac{1}{3} + 2K^2) Km^3\right) $, when only one iteration is required, \cref{alg:IBMI} is computationally faster compared with this direct method. For cases where \cref{alg:IBMI} takes more iterations to converge, it may be slower than direct inversion. Finally, we note that if the matrix $\Q$ has additional structure, this could be incorporated in the complexity analysis above. 

\section{Numerical Results} \label{sec:graphs}
Some numerical results to highlight the capabilities of \cref{alg:IBMI} will now be detailed. These experiments were run on a 2023 M3 MacBook Pro with 8-core CPU, 10-core GPU and 16-core Neural Engine, 16GB unified memory and 1TB SSD storage, running macOS 15.1.1, using MATLAB 2024a and OpenBLAS. (Experiments were also run with Apple's Accelerate BLAS and the results were qualitatively similar.) 
\\ \\ 
\renewcommand{\arraystretch}{1.75}
\begin{table}[!b]
\vspace{-0.75cm}
 \caption{Covariance kernels used to generate dense symmetric positive definite covariance matrices.}
    \centering
     \begin{tabular}{|c|c|c|} \hline
  \textbf{Kernel} & \textbf{Covariance Matrix }\\  \hline 
     Exponential &  $ \Q_{EXP}(\b{x},\b{x}')= \exp \left(\frac{-{\| \bf{x}-\bf{x}' \|}}{5}\right)$   \\ \hline
    RBF  &  $\Q_{RBF}(\b{x},\b{x}')= \exp\left(- \frac{\| \bf{x}-\bf{x}' \|^2}{2\sigma^2}\right) $ \\ \hline 
 Inverse Quadratic  & $ \Q_{IQUAD}(\b{x},\b{x}') = \frac{1}{\sqrt{1+ \| \bf{x}-\bf{x}'\|^2}}$   \\ \hline
 Matérn 3/2  & $\Q_{M3/2}(\b{x},\b{x}')= \left( 1 + \frac{\sqrt{3} \|\b{x} -\b{x}'\|}{\tau} \right) \exp\left( - \frac{\sqrt{3} \|\b{x} -\b{x}'\| }{\tau} \right)$ \\ \hline
 Matérn 5/2  & $\Q_{M5/2}(\b{x},\b{x}') = \left( 1 + \frac{\sqrt{5} \|\b{x} -\b{x}'\|}{\tau} + \frac{5 \|\b{x} -\b{x}'\|^2}{3 \tau^2} \right) \exp\left( - \frac{\sqrt{5} \|\b{x} -\b{x}'\|}{\tau} \right)$ \\ \hline
    \end{tabular}
  
    \label{tab:kernels}
\end{table}
Covariance matrices, $\Q \in \mathbb{R}^{p\times p}$, which are dense and guaranteed to be symmetric positive definite, were used for the following numerical results. Five covariance kernels were used to generate covariance matrices, which are presented in \cref{tab:kernels}. These are the exponential kernel (EXP), the radial basis function (RBF) kernel, the inverse quadratic function kernel (IQUAD) and the Matérn 3/2 (M3/2) and Matérn 5/2 (M5/2) kernels. \one{Experiments are presented for both 1D and 2D data.} For 1D data, the values of $x$ and $x'$ used to generate the covariance matrix from the corresponding kernel are equally-spaced values from 0 to $p^{0.9}$. This ensures that the condition number increases moderately with the dimension. 
\one{For the 2D data, regular grids on the unit square are created with equally-spaced values from 0 to  $p^{\frac{0.9}{2}}$, which results in matrices with a larger bandwidth than their 1D counterparts.} 
The error estimate used as the stopping condition in \cref{alg:IBMI} is shown in \cref{eq:error}, with a set tolerance of $10^{-8}$.
\\ \\
\one{The rest of this section presents as follows. First, \cref{sec:dim_vs_iters} investigates the time needed for \cref{alg:IBMI} to converge for both 1D and 2D data. The same covariance matrices are used in \cref{sec:error_vs_iters}, where the number of iterations and final error are tabulated. The rate of convergence is discussed in \cref{sec:num_vs_theory}, and the influence of the partitioning of the covariance matrices for 1D and 2D data is investigated in  \cref{sec:overlap}. We consider the effects of varying the hyper-parameters of the covariance kernels, the choice of initial guess, and the ordering of variables in the covariance matrix on the convergence of \cref{alg:IBMI} in \cref{sec:hyper-params}, \cref{sec:initial_guess} and \cref{sec:red_black} respectively.  Finally, we conclude this section by reviewing the properties which affect the convergence of \cref{alg:IBMI}.}

\begin{figure}[!hp]
\centering 
\subfloat[Exponential Kernel ($\Q_{EXP}$)]{%
  \includegraphics[clip,width=0.75\columnwidth]{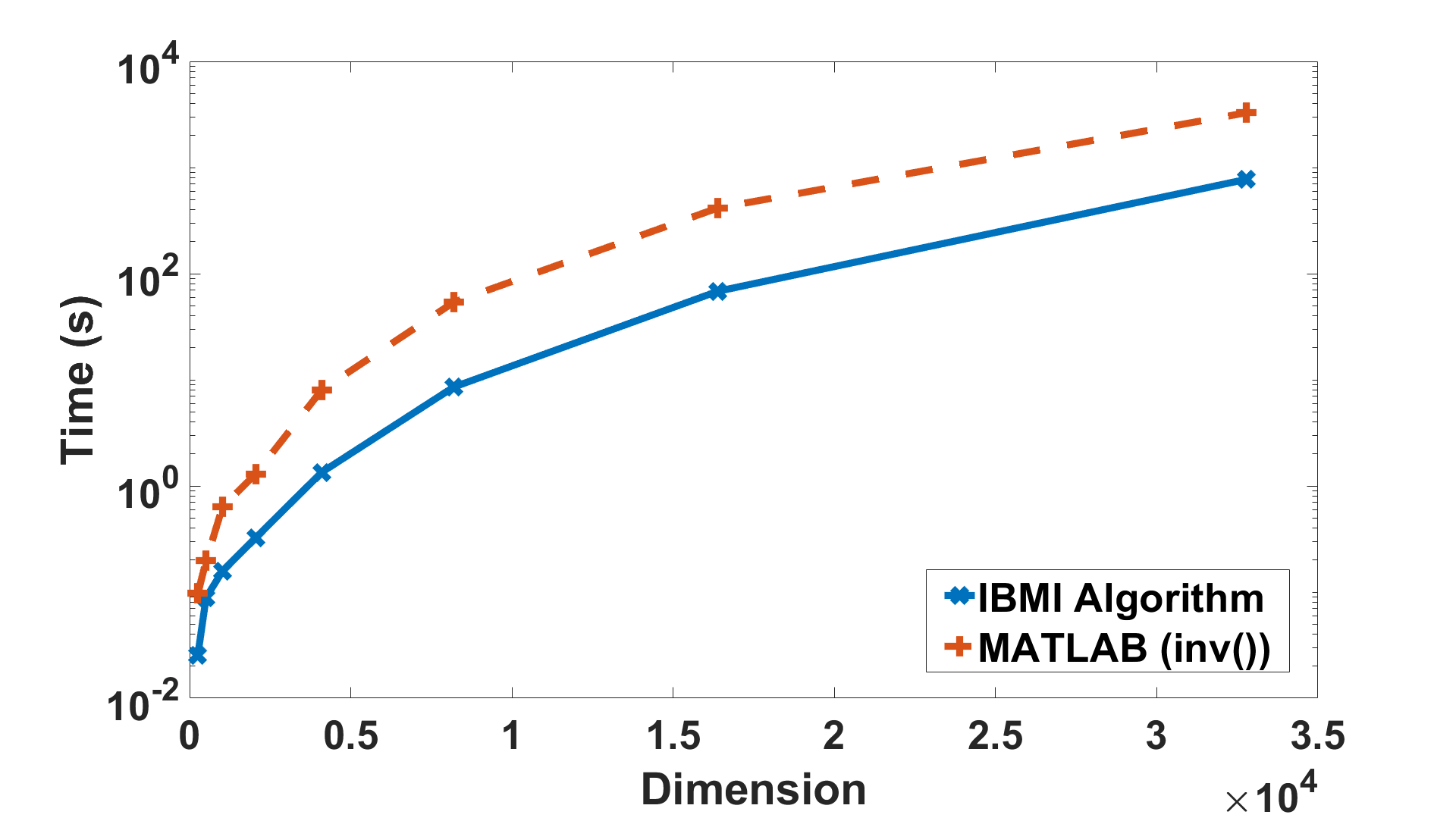}%
}

\subfloat[RBF Kernel ($\Q_{RBF}$)]{%
  \includegraphics[clip,width=0.75\columnwidth]{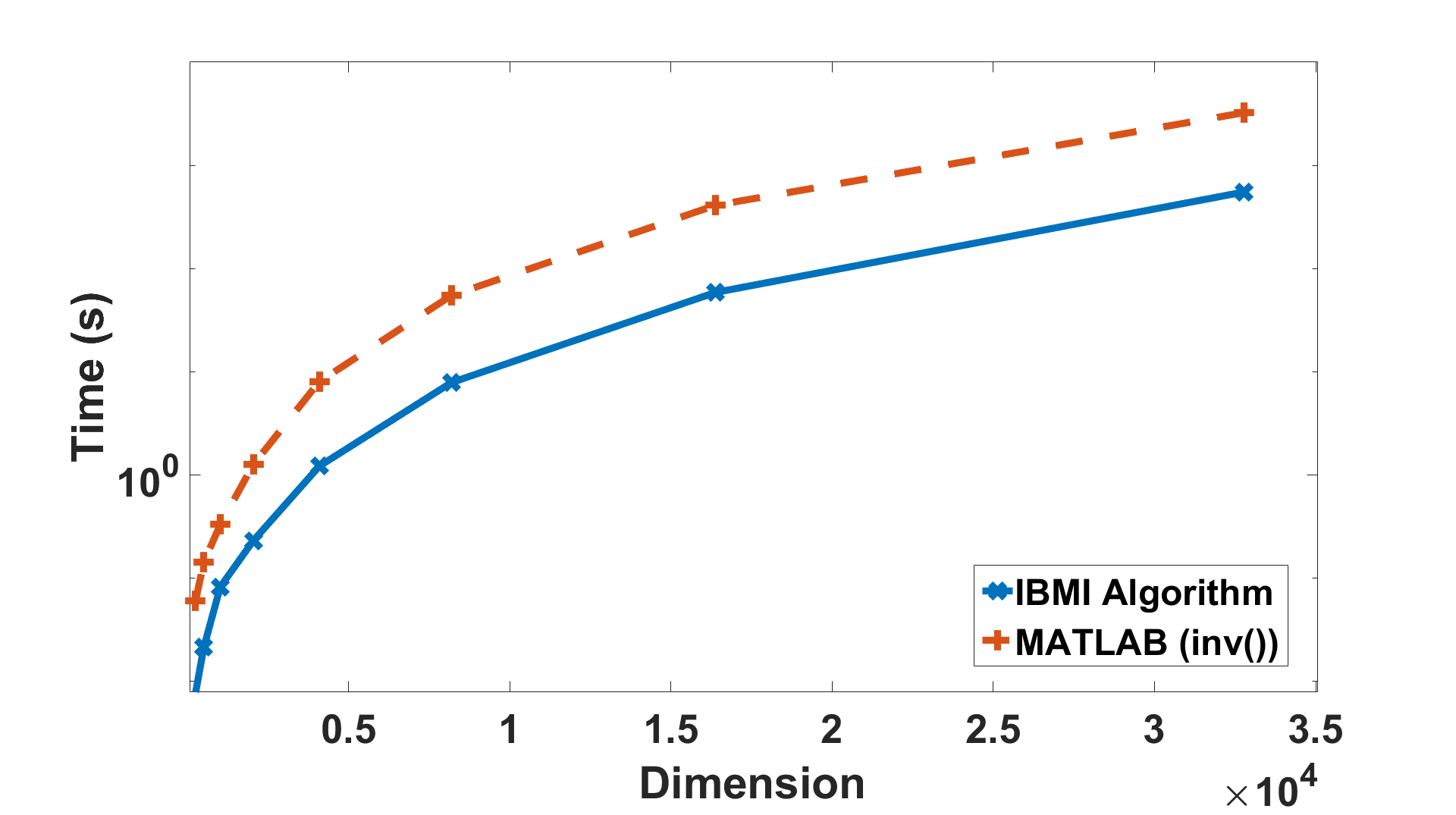}%
}

\subfloat[Inverse Quadratic Kernel ($\Q_{IQUAD}$)]{%
  \includegraphics[clip,width=0.75\columnwidth]{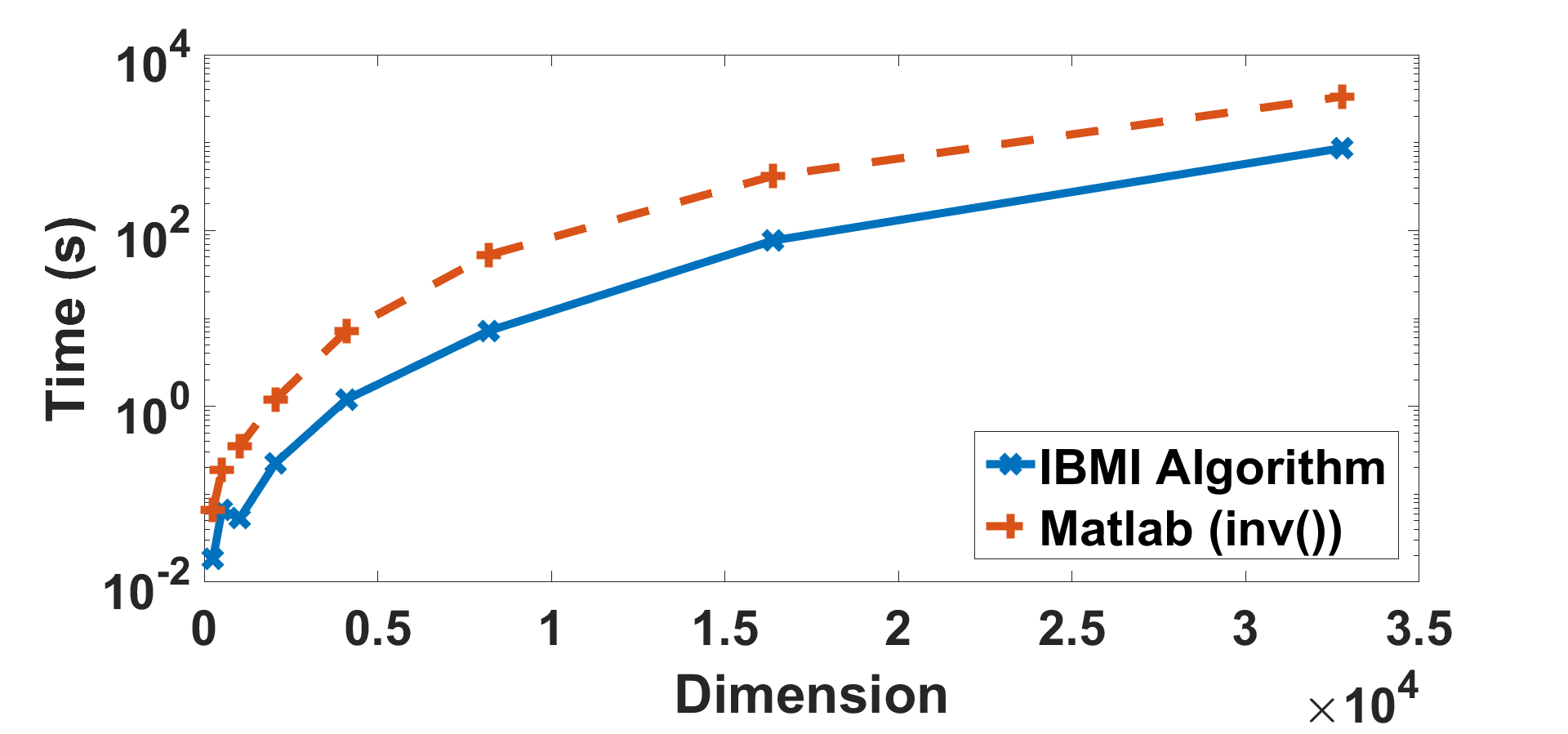}
}

\caption{Dimension of $\Q$ and the time taken for \cref{alg:IBMI} to converge and approximate $\tsig$, compared to the time taken by MATLAB's \texttt{inv()} function to compute $\sig$ \one{for 1D data}.} 
\label{Figure1}
\end{figure} 

\subsection{CPU Time} \label{sec:dim_vs_iters}
We first investigate the time taken for \cref{alg:IBMI} to converge, for different covariance matrices as the dimension $p$, of the matrices increases. Specifically, $p = 2^\ell$, where $\ell = 8,\dotsc, 15$. (Larger covariance matrices could not be stored.) 
The time taken for \cref{alg:IBMI} to approximate the inverse of each covariance matrix, generated by the \one{first three} kernels in \cref{tab:kernels}, was compared with the time taken for MATLAB's inverse function \texttt{inv()} to invert the same matrices. \one{The covariance matrices generated with 1D and 2D data were partitioned into two block rows and columns with a 20\% overlap.}

\subsubsection{1D Data} 
\one{It can be seen in \cref{Figure1} that \cref{alg:IBMI} converges faster for 
covariance matrices generated with 1D data irrespective of dimension, compared to the in-built function \texttt{inv()} for all three covariance kernels. 
For example, when $p=2^{10}$, \cref{alg:IBMI} took 
0.1584, 0.0816, and  0.0531
seconds for the covariance matrices generated by the exponential, RBF, and inverse quadratic kernels. 
On the other hand, MATLAB's inverse function took 
0.6377, 0.3282, and  0.3503
seconds for the same matrices. 
For the largest covariance matrices, \cref{alg:IBMI} converged 
in 772, 550 and 864
seconds for the covariance matrices generated by the exponential, RBF and inverse quadratic kernels, while MATLAB's inverse function took 
3 277 (EXP kernel),
3 234 (RBF kernel), and 3 329 (IQUAD kernel) seconds, respectively.} 

\one{\subsubsection{2D Data} 
For each covariance matrix generated with 2D data, \cref{fig:2D_dim_vs_time} shows that \cref{alg:IBMI} also converges faster than MATLAB's inverse function \texttt{inv}. For example, when $p=2^{10}$, \cref{alg:IBMI} took  
 0.1088, 0.1511,  and 0.1455
seconds for the RBF, exponential and inverse quadratic kernels compared to 
0.3600, 0.2250, and 0.2466
seconds when MATLAB's inverse function was applied to the same covariance matrices.  When $p=2^{14}$, the time taken for \cref{alg:IBMI} to converge was  
59.47, 64.17, and 129.84
seconds for the RBF, exponential and inverse quadratic covariance matrices, while MATLAB's inverse function took 
411.70 (RBF kernel), 402.35 (EXP kernel), and 401.56 (IQUAD)
seconds, respectively.}

\begin{figure}[!h]
\centering 
\subfloat[Exponential Kernel ($\Q_{EXP}$)]{%
  \includegraphics[clip,width=0.75\columnwidth]{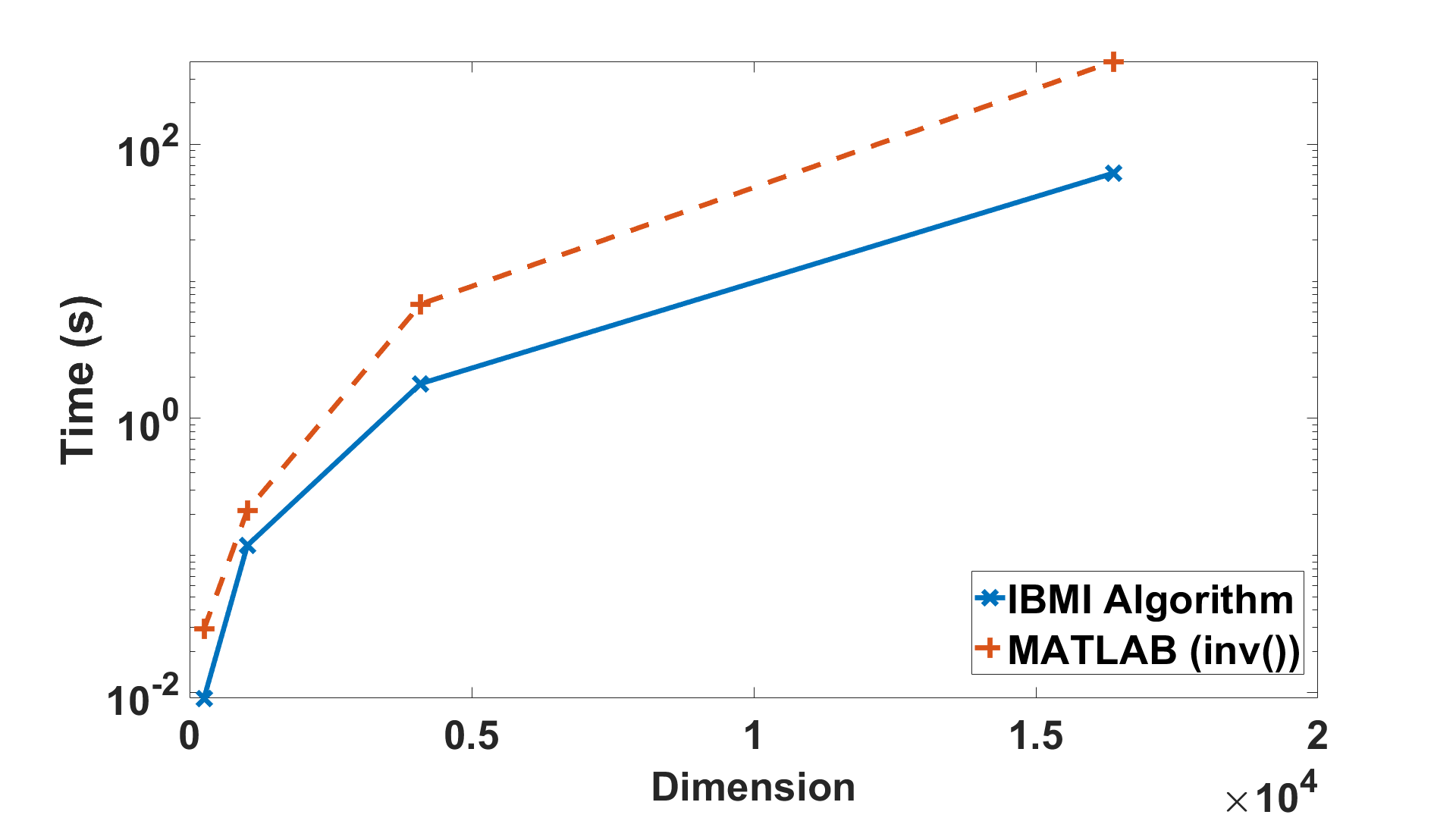}%
}

\subfloat[RBF Kernel ($\Q_{RBF}$)]{%
  \includegraphics[clip,width=0.75\columnwidth]{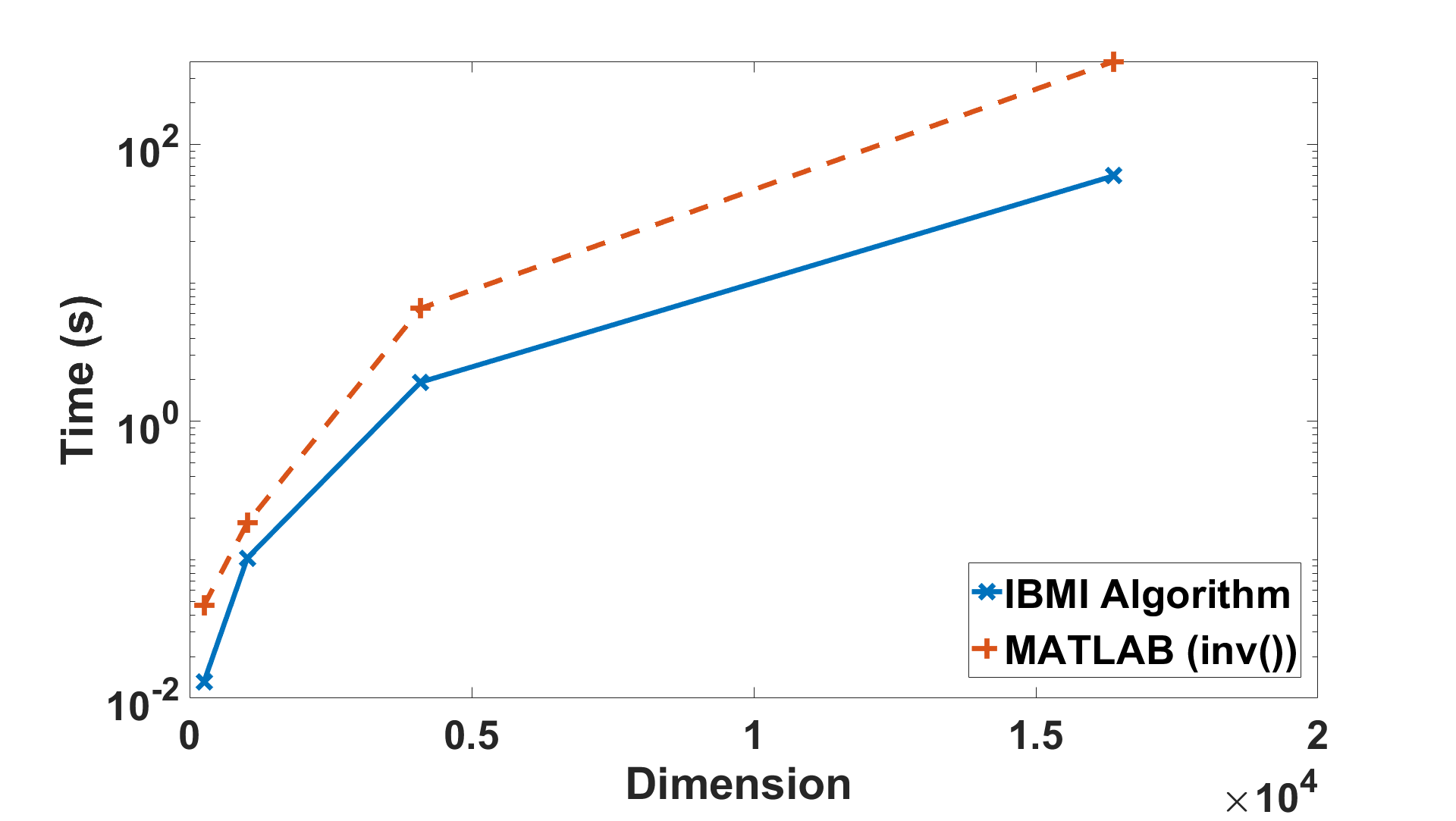}%
}
\vspace{-0.45cm}

\subfloat[Inverse Quadratic Kernel ($\Q_{IQUAD}$)]{%
  \includegraphics[clip,width=0.75\columnwidth]{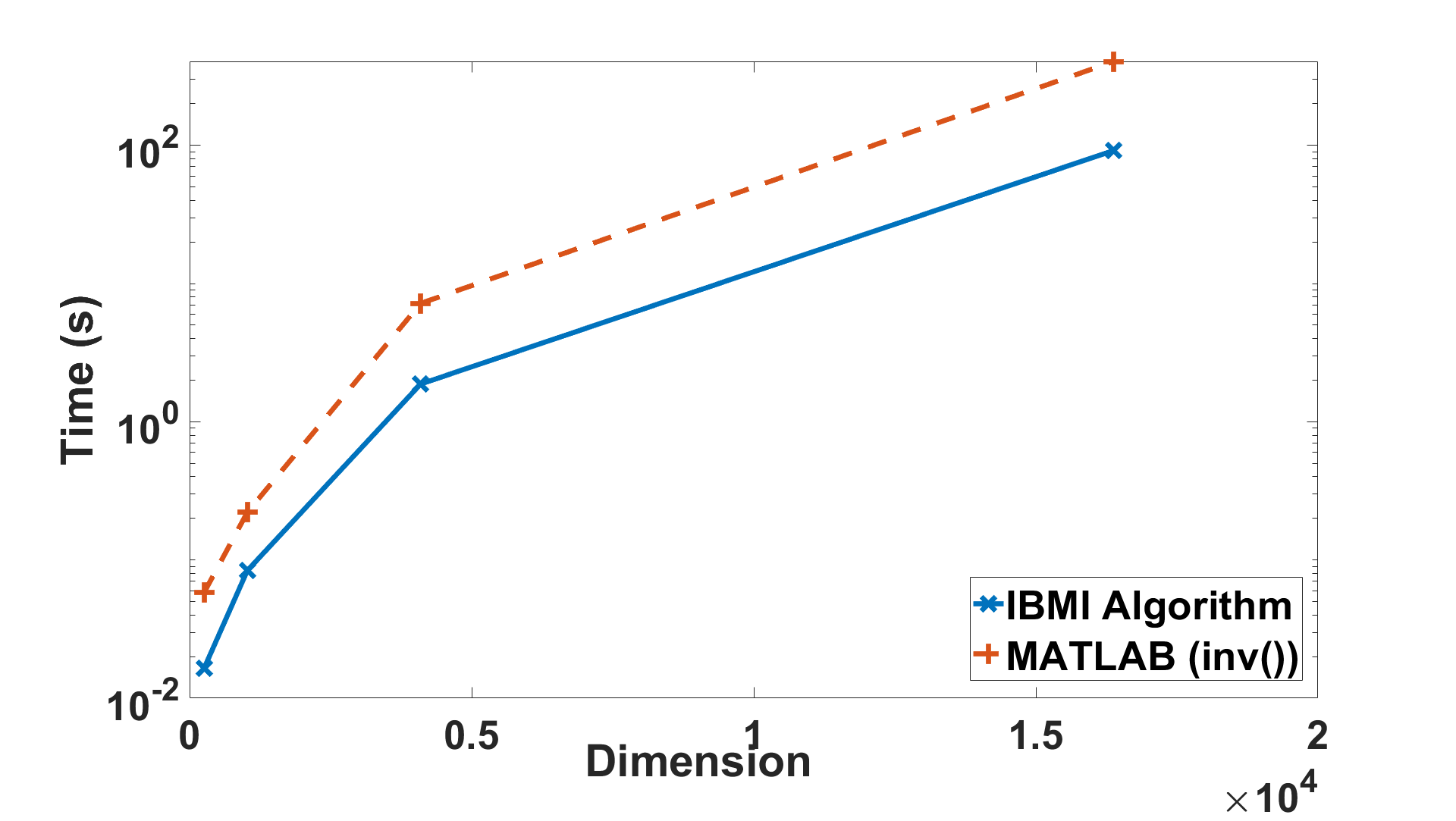}
}

\caption{\one{Dimension of $\Q$ and the time taken for \cref{alg:IBMI} to converge and approximate $\tsig$, compared to the time taken by MATLAB's \texttt{inv()} function to compute $\sig$ for 2D data.}} 
\label{fig:2D_dim_vs_time}
\vspace{-0.8cm}
\end{figure}

\renewcommand{\arraystretch}{1.1}
\begin{table}[b!]
\centering
\vspace{-0.5cm}
\caption{The number of iterations for \cref{alg:IBMI} to converge, and the error in $\tsig$, for matrices generated by the three covariance kernels in \cref{tab:kernels} for 1D and 2D data. } \label{tab:iters vs error}
\begin{tabular}{|l|l|lll|lll|}
\hline
Data                & Dim   & \multicolumn{3}{l|}{Number of Iterations} & \multicolumn{3}{l|}{Error $ \| \tsig - \sig \|_2$} \\ \hline
 &   & EXP  & RBF & IQUAD & EXP & RBF  & IQUAD  \\ \hline
\multirow{8}{*}{1D} & \textbf{$2^8$} & 1           & 1           & 1            & 8.8776e-13     & 2.1237e-15     & 5.4519e-11     \\
& $2^9$          & 1           & 1           & 1            &    2.2002e-12    &  4.5937e-15      & 5.1270e-12      \\
& $2^{10}$       & 1           & 1           & 1            & 1.5159e-12       &  1.4811e-14     & 3.0701e-12     \\
& $2^{11}$       & 1           & 1           & 1            &   2.046e-12    & 4.8692e-14       & 2.7929e-11     \\
& $2^{12}$       & 1           & 1           & 1            &  2.5946e-12     & 1.593e-13      & 1.3058e-10     \\
& $2^{13}$       & 1           & 1           & 1            &     5.5856e-12  & 2.2981e-13      & 3.3384e-10     \\
& $2^{14}$       & 1           & 1           & 1            &    5.6725e-12     &1.6997e-12      &  1.5243e-09     \\
& $2^{15}$       & 1           & 1           & 1            &  4.8877e-12              &    9.5423e-12              & -              \\ \hline
\multirow{4}{*}{2D} & \textbf{$2^8$} & 5  & 2   & 4  &  4.0129e-10   &    1.6277e-13   & 1.3557e-10  \\
 & $2^{10}$       & 2          & 3         & 1    &     5.1995e-11 &     7.3451e-11  &     8.9139e-10  \\
 & $2^{12}$       & 2          & 2         & 1  &   2.8948e-12 &  8.7409e-15     &     2.3953e-12   \\
 & $2^{14}$       & 1          & 1         & 1    &    5.7933e-10    &  2.786e-14   &   8.3444e-12  \\ \hline 
\end{tabular}
\end{table}

\subsection{Error vs Number of Iterations} \label{sec:error_vs_iters}
\cref{tab:iters vs error} displays the number of iterations taken for \cref{alg:IBMI} to converge, and $ \| \tsig - \sig \|_2$ at the last iteration, where $\| \cdot \|_{2}$ is the 2-norm,  as the dimension of the covariance matrix increases. Here, $\tsig$ is the inverse computed using MATLAB's \texttt{inv} function.
The error could not be computed for the covariance matrix of dimension $2^{15}$ generated by the IQUAD kernel, due to memory constraints.
We note that the errors in \cref{tab:iters vs error} differ from the residual-based measure used in the stopping criterion (cf.\ \cref{eq:error}), because $\sig$ is unknown in practice. 

\subsubsection{\two{1D Data}}
\two{All three covariance kernels took only one iteration to converge irrespective of the dimension of the covariance matrix generated with 1D data, as shown in \cref{tab:iters vs error}. The best approximated matrices came from the RBF covariance kernel for smaller dimensions. The covariance matrices produced by the exponential and inverse quadratic covariance kernels also had low errors for smaller covariance matrices. For each covariance kernel, the errors tend to increase with the dimension of the covariance matrix. }



\subsubsection{\two{2D Data}} 
\two{For the covariance matrices generated with 2D data, \cref{tab:iters vs error} shows that \cref{alg:IBMI} takes between one and five iterations to converge, depending on the dimension of the covariance matrix. The number of iterations decreased for the exponential and inverse quadratic kernels as the dimension of the covariance matrices increased. The RBF kernel also saw a decrease the number of iterations for covariance matrices larger than $2^{10}$ in dimension. Here the errors are uniformly small, and do not appear to increase with the dimension. The best approximated matrix generated with 2D data came from the RBF kernel again, when $p=2^{12}$. }

\subsection{Numerical vs theoretical convergence rate} \label{sec:num_vs_theory}
\begin{figure}[b!]
    \centering
    \includegraphics[width=0.8\linewidth]{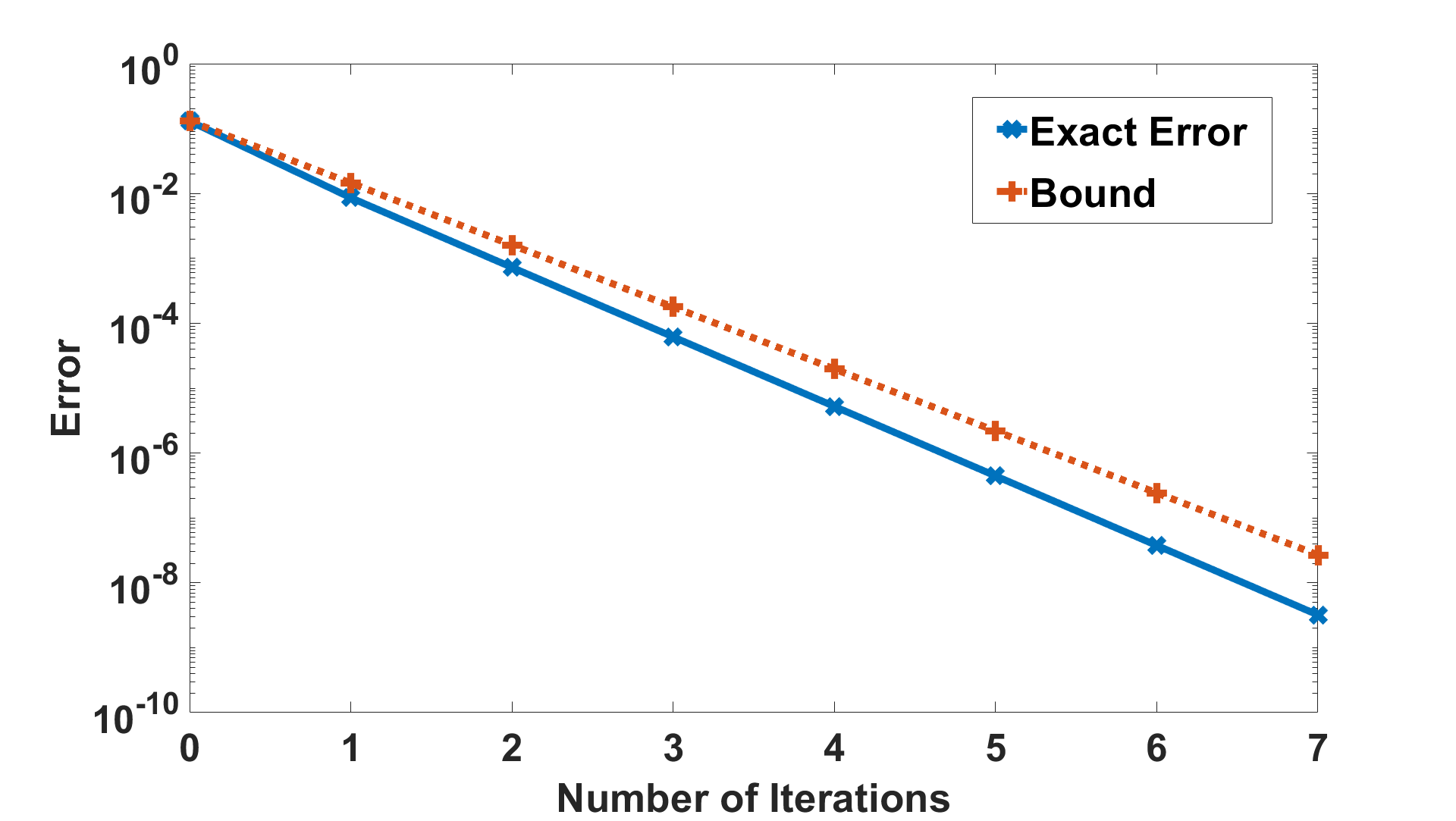}
    \caption{Comparison of the error $\|\tsig_{\I_2}^{(r,2)} - \sig_{\I_2}\|_2$ for \cref{alg:IBMI} with two 
    non-overlapping blocks and the error bound in \cref{lemma} for  $\Q_{RBF}$ of dimension $p = 2^{12}$.} \label{Figure3}
    \vspace{-0.8cm}
\end{figure}
Given any symmetric positive definite matrix, \cref{lemma} guarantees that \cref{alg:IBMI} will converge when $\Q$ is partitioned using two, non-intersecting sets. Moreover, it provides the upper bound \cref{eq:bound} on the error reduction at each iteration. We examine whether this bound is descriptive for a covariance matrix generated using the RBF kernel of dimension $2^{12}$. 
As the number of iterations increases, \cref{Figure3} confirms that the actual error decreases linearly, similarly to the upper bound \cref{eq:bound}. The convergence rate is better, but fairly similar to, the rate of 0.333 predicted by the bound, indicating that the bound is reasonably descriptive in this case. 
\one{Although we do not have theoretical convergence bounds for other partitionings, we investigate the convergence rate of \cref{fig:conv_rate} of a multi-block, non-overlapping partitioning in \cref{fig:conv_rate}. 
For all dimensions of $\Q_{RBF}$ the error decreases linearly as the number of iterations increases. However, the error decreases at a slower rate as the dimension of  $\Q_{RBF} $ increases. }

\begin{figure}[t!]
\vspace{-0.2cm}
    \centering 
    \includegraphics[width=0.8\linewidth]{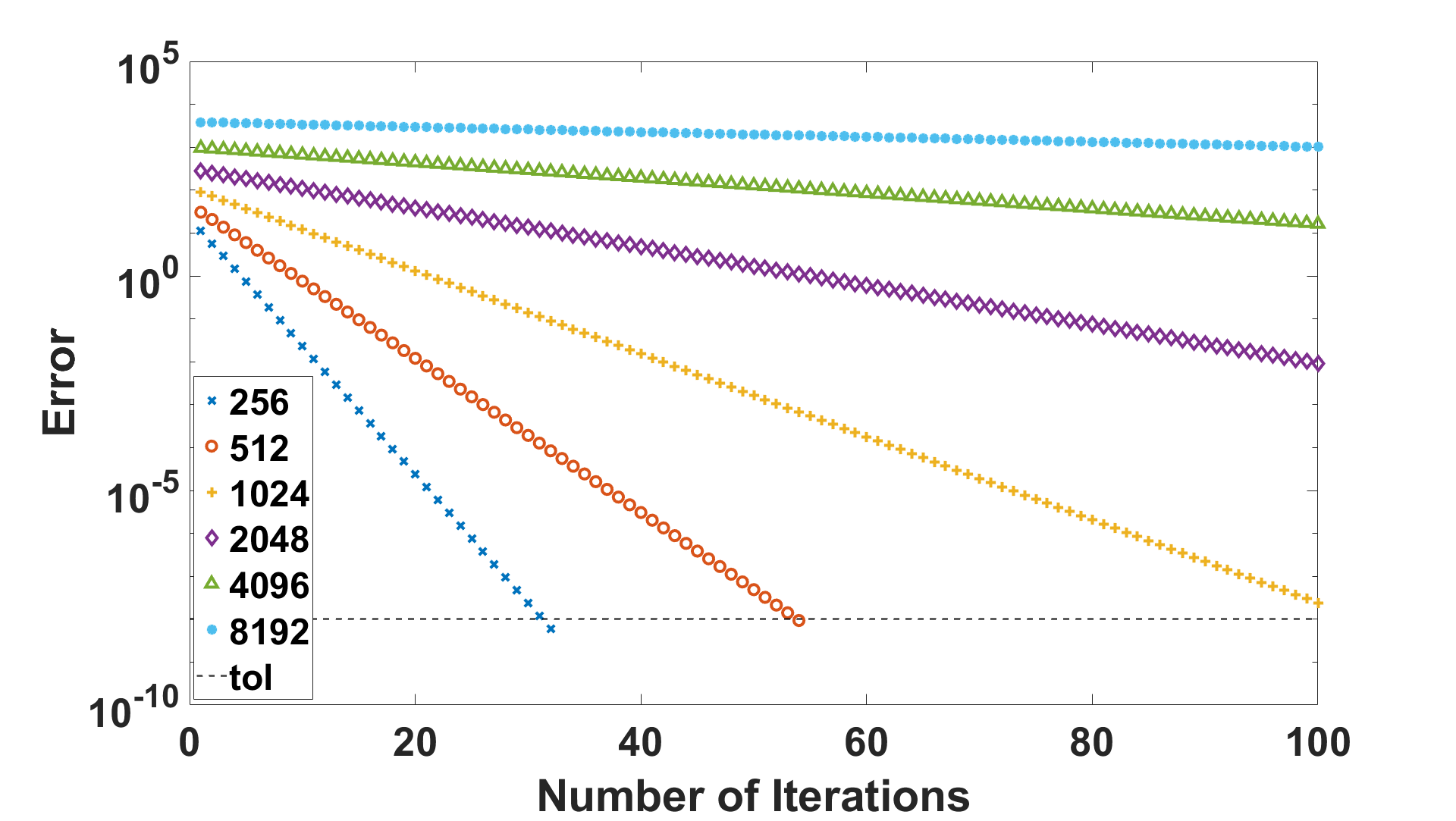}
    \caption{Comparison of the error $\|\tsig_{\I_2}^{(r,2)} - \sig_{\I_2}\|_2$ for \cref{alg:IBMI} with four 
    non-overlapping blocks for $\Q_{RBF}$ of dimension $p = 2^{\ell},$ where $\ell=8,\ldots,14$. \label{fig:conv_rate}}
    \label{fig:placehold}
    \vspace{-0.3cm}
\end{figure}

\subsection{Influence of the Partitioning on the Convergence\label{sec:overlap}}

The partitioning of the covariance matrix $\Q$ can greatly affect the convergence rate of \cref{alg:IBMI}.  \cref{lemma} details how  \cref{alg:IBMI} will converge for the two-block non-overlapping partitioning, given any symmetric positive definite matrix $\Q$. Here some numerical results are displayed which suggests that the multi-block partitioning with overlap will converge faster than the two-block partitioning. \one{Both 1D and 2D} covariance matrices $\Q$, of dimension $2^{12}$ were generated using the RBF covariance kernel. When partitioning the matrix for \cref{alg:IBMI}, the number of blocks was varied between 2 and 6, while the amount of overlap varied between 0\% and 20\%. The effect on the time taken for \cref{alg:IBMI} to converge, and the number of iterations required, was then recorded.

\renewcommand{\arraystretch}{1.4}
\newcommand*\rot{\rotatebox{90}}
\begin{table}[b!]
\vspace{-0.5cm}
\centering
\caption{Time taken for \cref{alg:IBMI} to \two{converge} in seconds by altering the number of blocks and overlap between the blocks, when partitioning a covariance matrix $\Q_{RBF}$ of dimension $p=2^{12}$ generated with {1D data}.} \label{tab:blocks & overlap}
\begin{tabular}{ll|lllll}
\multicolumn{2}{c|}{}   & \multicolumn{5}{c}{Overlap Fraction}  \\  
 &   & 0.00   & 0.05 & 0.10    & 0.15   & 0.20 \\ \hline
\multirow{5}{*}{\rot{Number of Blocks}}  & 2 & 323.180 & 1.1419 & 1.1054 & 1.1807 & 1.0588 \\
 & 3 & 398.898 & 1.1575 & 1.1486 & 1.1673 & 1.1606 \\
& 4 & 401.556 & 1.1791 & 1.1724 & 1.2099 & 1.2433 \\
& 5 & 469.355 & 1.2222 & 1.2472 & 1.2666 & 1.315  \\
& 6 & 487.718 & 1.3113 & 1.2535 & 1.2749 & 1.3823 \\ \hline
\multicolumn{2}{l|}{Iters} & 476 & 1  & 1      & 1 & 1 
\end{tabular}
\vspace{-0.5cm}
\end{table}

\subsubsection{1D Data}
The results for the 1D problem are given in \cref{tab:blocks & overlap}. 
Note that the number of iterations was independent of the number of blocks, $K$, within the tested range of $K=2,\dotsc, 6$.
\cref{tab:blocks & overlap} illustrates how even a small amount of overlap greatly decreased the number of iterations, and hence time, for \cref{alg:IBMI} to converge. When non-overlapping blocks were used,  \cref{alg:IBMI} took 476 iterations to converge, taking between 323 seconds (for two blocks) and 488 seconds (for six blocks). However, by introducing only a  5\% overlap, the algorithm converged in 1 iteration and between 1.14 and 1.31 seconds.  
\\ \\ 
When no overlap is used, it was quicker to use a two block partitioning with Algorithm \eqref{alg:IBMI}. 
When overlap was introduced, only one iteration was required and the timings were very similar for all choices of $K$. 
The time for \cref{alg:IBMI} to converge increased slightly with the number of blocks and the overlap fraction and, for this particular matrix, 
the smallest time was achieved for two blocks and a 10\% overlap.
However, the variation in timings for the overlapping cases was small, indicating that the algorithm is fairly insensitive to the number of blocks in the partitioning, and the amount of overlap. Although our default choices in other experiments are $K=4$ blocks and an overlap of 5\%, results are fairly similar for other partitionings. 
\begin{table}[b!]
\centering
\caption{ Time taken for \cref{alg:IBMI} to \two{converge} in seconds and iterations (in parentheses) by altering the number of blocks and overlap between the blocks, when partitioning a covariance matrix $\Q_{RBF}$ of dimension $p=2^{12}$ generated with {2D data}.\label{tab:2Dblocks & overlap}}
\begin{tabular}{ll|lllll}
&  & \multicolumn{5}{c}{Overlap Fraction}                   \\ \cline{3-7} 
                &         & 0.00          & 0.05        & 0.10        & 0.15        & 0.20        \\ \hline
\multirow{5}{*}{\rot{Number of Blocks}} & 2 & 23.016 (33)   & 7.887 (11)  & 3.608 (5)   & 2.3125 (3)  & 1.662 (2) \\
&3 & 222.18 (257) & 11.606 (13) & 6.224 (7)   & 4.579 (5)  & 2.828 (3)  \\
&  4 & 29.472 (33)   & 28.658 (32) & 9.887 (11)  & 5.714 (6)  & 4.702 (5)   \\
&5 & 251.16 (257) & 32.121 (33) & 10.904 (11) & 10.709 (11) $\! \!$ & 6.239 (6)  \\
 & 6 & 264.14 (257) & 33.662 (33) & 13.435 (13) & 11.345 (11) $\! \!$& 7.365 (7)  \\ \hline
\end{tabular}
\end{table}

\subsubsection{\two{2D Data}}
\two{For the 2D problem, the number of iterations was not independent of the number of blocks (see \cref{tab:2Dblocks & overlap}). 
However, similarly to the 1D problem, when no overlap was used, it was quicker to use a two block partitioning. When overlap was introduced, the time taken for \cref{alg:IBMI} to converge, and the number of iterations decreased as the overlap increased, and the fastest time of 1.662 seconds was achieved when $\Q_{RBF}$ was partitioned into two block rows/columns with a $20\%$ overlap.} 
\\ \\ 
\two{Overall, \cref{tab:blocks & overlap} and \cref{tab:2Dblocks & overlap} highlight how introducing overlap appears to be more effective than optimising the number of blocks when partitioning the covariance matrix to achieve faster convergence for \cref{alg:IBMI}. 
For matrices with a larger bandwidth such as those generated with 2D data, it appears that using a two-block partitioning with minimum 20$\%$ overlap can be the most effective partitioning for \cref{alg:IBMI} to converge. However, by increasing the amount of overlap with two blocks, the sub-matrix $\Q_{\I}^{-1}$ can become computationally expensive to calculate, so introducing more blocks may be needed as the dimension of $\Q$ increases. }

\subsection{\two{Varying Covariance Kernel Hyper-parameters}}\label{sec:hyper-params} \two{
We now investigate the effect of covariance kernel hyper-parameters on the convergence of 
\cref{alg:IBMI} for covariance matrices generated using the RBF and the Matérn 3/2 kernels in \cref{tab:kernels}. In both cases the matrices were of dimension $2^{12}$ and were partitioned into four blocks with a 5\% overlap.}
\begin{table}[t!]
\vspace{-0.5cm}
\caption{\two{Varying hyper-parameters $\sigma$ and $\tau$ of the RBF and Matérn 3/2 covariance kernels respectively to see how this affects the convergence of \cref{alg:IBMI} and the condition number of $\Q$.} \label{tab:hyperparameters}}
\begin{tabular}{lccccl}
\hline
\multicolumn{1}{|l|}{Kernel}                                                                 & \multicolumn{1}{l|}{$\sigma$} & \multicolumn{1}{l|}{No of Iterations} & \multicolumn{1}{l|}{Time (seconds)} & \multicolumn{1}{l|}{Error} & \multicolumn{1}{l|}{Cond(A)}    \\ \hline
\multicolumn{1}{|l|}{\multirow{5}{*}{RBF}}                                                   & 0.3                           & 1                                     & 1.3854                              & 2.4118e-16                 & \multicolumn{1}{l|}{5.2071}     \\
\multicolumn{1}{|l|}{}                                                                       & 0.5                           & 1                                     & 1.2875                              & 9.8146e-15                 & \multicolumn{1}{l|}{335.3515}   \\
\multicolumn{1}{|l|}{}                                                                       & 0.7                           & 1                                     & 1.2450                              & 8.0230e-11                 & \multicolumn{1}{l|}{1.7337e+05} \\
\multicolumn{1}{|l|}{}                                                                       & 0.9                           & 500 (max)                             & 475.0415                            & 2.2618e-04                 & \multicolumn{1}{l|}{7.1930e+08} \\
\multicolumn{1}{|l|}{}                                                                       & 1.1                           & 500 (max)                             & 476.7224                            & 2.7833e+06                 & \multicolumn{1}{l|}{2.3942e+13} \\ \hline
\multicolumn{1}{c}{}                                                                         &                               &                                       &                                     &                            & \multicolumn{1}{c}{}            \\ \hline
\multicolumn{1}{|c|}{Kernel}                                                                 & \multicolumn{1}{l|}{$\tau$}   & \multicolumn{1}{l|}{No of Iterations} & \multicolumn{1}{l|}{Time (seconds)} & \multicolumn{1}{l|}{Error} & \multicolumn{1}{l|}{Cond(A)}    \\ \hline
\multicolumn{1}{|l|}{\multirow{5}{*}{\begin{tabular}[c]{@{}l@{}}Matérn \\ 3/2\end{tabular}}} & 3                             & 1                                     & 1.2552                              & 6.0534e-13                 & \multicolumn{1}{l|}{1.275e+04}  \\
\multicolumn{1}{|l|}{}                                                                       & 6                             & 1                                     & 1.2064                              & 1.6069e-11                 & \multicolumn{1}{l|}{1.9296e+05} \\
\multicolumn{1}{|l|}{}                                                                       & 9                             & 1                                     & 1.1891                              & 9.3210e-10                 & \multicolumn{1}{l|}{9.7505e+05} \\
\multicolumn{1}{|l|}{}                                                                       & 12                            & 1                                     & 1.2119                              & 6.3821e-10                 & \multicolumn{1}{l|}{3.0794e+06} \\
\multicolumn{1}{|l|}{}                                                                       & 15                            & 500 (max)                             & 474.7873                            & 1.0884e-08                 & \multicolumn{1}{l|}{7.5146e+06} \\ \hline
\end{tabular}
\end{table}
\\ \\
\two{
We first consider the covariance matrices generated by the RBF kernel. The length scale parameter $\sigma$  varied 0.3 and 1.1; for larger values of $\sigma$, $\Q$ was numerically singular.  
We see from \cref{tab:hyperparameters} that for $\sigma \le 0.7$, \cref{alg:IBMI} converges in a single iteration, in a time of just over 1 second. For larger values of $\sigma$, however, \cref{alg:IBMI} failed to converge within 500 iterations. This indicates that the conditioning of the matrix may affect the convergence rate of the IBMI algorithm.} 
\\
\\
\two{
\cref{tab:hyperparameters} also shows results for the  Matérn 3/2 kernel when the length scale parameter $\tau$ varied between 3 and 15. For $\tau < 15$, \cref{alg:IBMI} converged in just over one second, and in one iteration. For larger $\tau$ values, \cref{alg:IBMI} did not converge within 500 iterations, again indicating that the conditioning of the matrix can affect the algorithm's convergence rate. However, increasing the amount of overlap, reducing the number of blocks can improve performance (cf. \cref{sec:overlap}), as can reducing the tolerance in applications in which a less accurate approximation is acceptable.} 

\subsection{\two{Initial Guess}\label{sec:initial_guess}}
\two{In the previous experiments, the initial guess $\tsig_{\I_1^c}^{(0,1)}$ in \cref{alg:IBMI} was the identity matrix. 
We now investigate whether the initial guess has a significant impact on the convergence rate of \cref{alg:IBMI}. Four different initial guesses were considered: the identity matrix, $\Q_{\I^c}^{-1}$, the Monte Carlo estimator from \cref{MC Estimators} and the Takahashi equations \cref{eq:taka1} to obtain an approximation of  the inverse Schur complement $\sig_{\I^c}^{-1}$ to use as the initial guess. We applied \cref{alg:IBMI} with these initial guesses for covariance matrices of size $p=2^{12}$, generated by the exponential, RBF, inverse quadratic and Matérn 5/2 kernels. These matrices were partitioned into two non-overlapping block rows and columns. }
\\ \\
\two{The numbers of iterations required for \cref{alg:IBMI} to converge are shown in \cref{tab:initial_guess}. We see that the initial guess has very little impact on the number of iterations needed in this case, with the possible exception of the Matérn 5/2 kernel, for which 4 more iterations were needed when the identity matrix was used than for any other initial guess. 
It does not appear that one initial guess stands out as optimal when looking at the number of iterations it takes for  \cref{alg:IBMI} to converge. We continue to use the identity matrix as the initial guess, as it is less computationally expensive to generate than the other three options.}
\\ \\ 
\one{For the above hyper-parameter and initial guess experiments, the matrices were partitioned using four blocks with a 5\% overlap and two blocks with no overlap respectively. One way to improve the convergence of \cref{alg:IBMI} with these methods would be to add more overlap and reduce the number of blocks if memory permits as suggested by \cref{sec:overlap}. }

\begin{table}[b!]
\centering
\vspace{-0.75cm}
\caption{\two{The effect of the initial guess $\tsig_{\I^c}$ on the convergence rate of \cref{alg:IBMI} for covariance matrices of dimension $p=2^{12}$ formed from various kernels.}} \label{tab:initial_guess}
\begin{tabular}{l|cccc|}
\cline{2-5}
& \multicolumn{4}{c|}{Number of Iterations } \\ \hline
\multicolumn{1}{|l|}{Initial Guess $\backslash$ Kernel}    & EXP  & RBF  & IQUAD & M 5/2 \\ \hline
\multicolumn{1}{|l|}{Identity Matrix}  & 42   & 28   & 24    & 11    \\
\multicolumn{1}{|l|}{$\Q^{-1}_{\I^c}$} & 41   & 27   & 23    & 7    \\
\multicolumn{1}{|l|}{MC Estimator}     & 42   & 28   & 24    & 7    \\
\multicolumn{1}{|l|}{Takahashi}        & 41   & 27   & 23    & 7    \\ \hline
\end{tabular}
\vspace{-0.5cm}
\end{table}

\begin{figure}[t!]
    \centering  
    \includegraphics[width=0.8\linewidth]{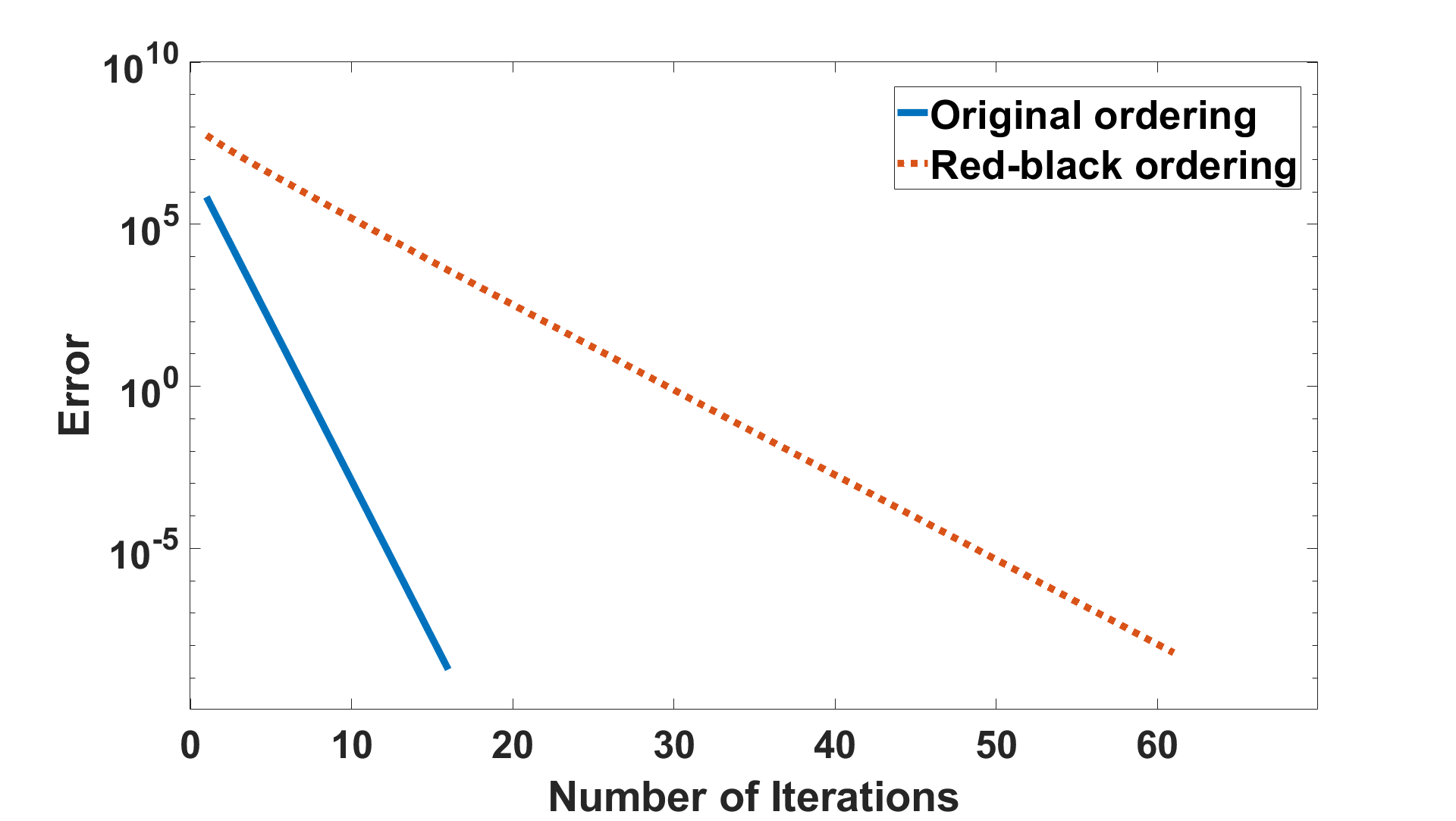}
    \caption{
    \two{Comparison of the error for different choices of the index sets $\I_1$ and $\I_2$ when the covariance matrix $\Q_{M5/2}$ of dimension $p=2^{12}$ is partitioned into two non-overlapping block rows/columns. The original (contiguous) ordering is compared to a red-black ordering. }   
    }
    \label{fig:red-black}
    \vspace{-0.75cm}
\end{figure}
\subsection{\two{Re-ordering}\label{sec:red_black}}
\two{We now assess whether the ordering of the rows and columns of $\Q\in\mathbb{R}^{p\times p}$ into the index sets $\I_k$ for $k=1,\ldots K$ changes the convergence rate of \cref{alg:IBMI}. By default, the index sets $\I_k$ are made from contiguous integers, e.g., in the case of two non-overlapping blocks of equal size, and with the dimension $p$ an even integer, we set $\I_1 = \{1,\dotsc, \frac{p}{2}\}$ and $\I_2 = \{\frac{p}{2}+1,\dotsc, p\}$.  
Here, for covariance matrices of dimension $p=2^{12}$, we instead partition $\Q$ with a red-black ordering, so that the index set $\I_1$ contains only odd indices and $\I_2$ contains only even indices. }
\\ \\ 
\two{For all five kernels in \cref{tab:kernels}, \cref{alg:IBMI} converged faster when the original (contiguous) index sets were used than when the red-black indexing was applied. Indeed, for every kernel except the Matérn 5/2 kernel, \cref{alg:IBMI} did not converge within 500 iterations with this alternative indexing. The convergence rates for the  Matérn 5/2 kernel are given in \cref{fig:red-black}, from which we see that with the red-black indexing \cref{alg:IBMI} took 61 iterations to converge to a final error of 5.9844e-09. When the original indexing was used, 16 iterations were required to obtain an error of 1.8197e-09. }
\\ \\ 
\two{
These results seem to suggest that, similarly to direct methods, it can be helpful to reorder the rows and columns of $\Q$ to obtain a more diagonally-dominant matrix with a smaller bandwidth. We note that other choices of indexing could potentially have a different effect on the convergence of \cref{alg:IBMI} but this has not been rigorously tested. However, all indices must be in at least one set $\I_k$, $k = 1,\dotsc K$ to guarantee convergence of \cref{alg:IBMI}; indeed, if even one is removed the method may fail to converge.  }





\subsection{\two{Properties Affecting Convergence}\label{sec:properties_conv}} 
\two{The above numerical experiments indicate that there are many factors that can affect the rate of convergence of \cref{alg:IBMI}, which we now discuss. Firstly, we saw in \cref{sec:dim_vs_iters} that \cref{alg:IBMI} usually converged faster for covariance matrices generated with 1D data than for those generated from 2D data. The 1D matrices have a smaller bandwidth and their elements decay faster away from the main diagonal than their 2D counterparts. The effect of these properties on the convergence rate can be most readily understood for the case of two non-overlapping blocks, since  \cref{lemma} indicates that convergence should be faster when off-diagonal blocks are smaller in norm. Numerical evidence also suggests that this is also true for multi-block and overlapping partitionings. 
Therefore, we postulate that improving diagonal dominance, when this is possible, can accelerate the convergence rate of \cref{alg:IBMI}}. 
\\ \\
\two {
\Cref{sec:hyper-params} shows that \cref{alg:IBMI} takes longer to converge as the condition number of $\Q$ increases. Improving the conditioning, e.g., through scaling, is therefore recommended when possible. 
However, it appears from \cref{sec:initial_guess} that the choice of initial guess is less important. That said, if a good approximation to the Schur complement is readily available, it could still be beneficial to use it.
}
\\ \\
\two{
The partitioning of the rows/columns of $\Q$ into sets had a significant impact on the performance of \cref{alg:IBMI}. We first saw in  \cref{sec:overlap} that introducing overlap is imperative for fast convergence, and that even a small amount of overlap can make a significant difference to the running time of the algorithm. Optimising the number of blocks can also improve performance, although less markedly. For covariance matrices with a larger bandwidth, the optimal partitioning appears to be into two blocks  with a decent overlap, e.g., $20\%$. Additionally, \cref{sec:red_black} showed that the ordering of the index sets $\I_k$ for $k=1, \ldots K$ is important.   Given that red-black ordering destroyed the fast decay of elements away from the diagonal and caused slower convergence, care should be exercised when choosing the sets $\I_k$, $k = 1,\dotsc, K$. We recommend choosing sets to maximise the `weight' of elements near the diagonal. 
}


\section{Discussion} \label{sec:conclusions}
In this paper, we have presented a novel iterative block matrix inversion algorithm which can accurately and efficiently approximate the inverse of a dense symmetric positive definite matrix. The IBMI algorithm serves as a way to approximate the off-diagonal elements of the inverse of a symmetric positive definite matrix, which is a known limitation for current literature. When $\Q$ is partitioned into two non-intersecting sets, \cref{alg:IBMI} will always converge, as shown in \cref{lemma}. Numerical results indicate that the multi-block partitioning with overlap accelerates the convergence of \cref{alg:IBMI}. Moreover, \cref{alg:IBMI} outperforms MATLAB's built-in inverse function, \texttt{inv()} in terms of time and computational complexity for the large dense matrices examined in \cref{sec:graphs}. \\ \\ 
\cref{alg:IBMI} is generally applicable to any symmetric positive definite matrix, without any additional constraints such as converting $\Q$ into a hierarchical low rank matrix and therefore, has the potential to assist with a wide range of modern problems within data science, machine learning and multivariate statistics. One application which could benefit significantly is Gaussian process regression (GPR), as both the covariance matrix and its inverse (the precision matrix) are needed for prediction and uncertainty quantification. For high-dimensional data sets, directly inverting the covariance matrix to derive the posterior predictive equations can become computational infeasible. \cref{alg:IBMI} could offer a potential solution for obtaining the inverse, allowing GPR to be applied to these problems. 
Furthermore, the IBMI algorithm could potentially be altered to approximate block diagonal sub-matrices of, $\tsig$ rather than the full matrix. This partial approximation may be beneficial to methods where only a subset of the full inverse is required.  

\section*{Acknowledgments}
We would like to acknowledge Professors Finn Lindgren and John Pearson for helpful discussions. We would also like to thank the anonymous referees for their helpful comments and suggestions. 

\bibliographystyle{siamplain}
\bibliography{references}

\begin{thebibliography}{10}

\bibitem{GPR_application}
{\sc S.~Ambikasaran, D.~Foreman-Mackey, L.~Greengard, D.~W. Hogg, and M.~O’Neil}, {\em Fast direct methods for {G}aussian processes}, IEEE Transactions on Pattern Analysis and Machine Intelligence, 38 (2015), pp.~252--265, \url{https://doi.org/10.1109/TPAMI.2015.2448083}.

\bibitem{bebendorf}
{\sc M.~Bebendorf}, {\em Hierarchical Matrices}, Springer Berlin Heidelberg, 2008, pp.~49--98, \url{https://doi.org/10.1007/978-3-540-77147-0_3}.

\bibitem{multi_Schwarz}
{\sc M.~Benzi, A.~Frommer, R.~Nabben, and D.~B. Szyld}, {\em Algebraic theory of multiplicative schwarz methods:}, Numerische Mathematik, 89 (2001), p.~605–639, \url{https://doi.org/10.1007/s002110100275}.

\bibitem{ChowandSaad}
{\sc E.~Chow and Y.~Saad}, {\em Preconditioned {K}rylov subspace methods for sampling multivariate gaussian distributions}, SIAM Journal on Scientific Computing, 36 (2014), pp.~A588--A608, \url{https://doi.org/10.1137/130920587}.

\bibitem{Duff_chpt3}
{\sc I.~S. Duff, A.~M. Erisman, and J.~K. Reid}, {\em Gaussian Elimination for Dense Matrices: The Algebraic Problem}, Oxford University, 2~ed., Jan. 2017, p.~43–61, \url{https://doi.org/10.1093/acprof:oso/9780198508380.003.0003}.

\bibitem{Erisman.Taka}
{\sc A.~M. Erisman and W.~F. Tinney}, {\em On computing certain elements of the inverse of a sparse matrix}, Commun. ACM, 18 (1975), p.~177–179, \url{https://doi.org/10.1145/360680.360704}.

\bibitem{stochastic_opti}
{\sc A.~Godichon-Baggioni, W.~Lu, and B.~Portier}, {\em Online estimation of the inverse of the hessian for stochastic optimization with application to universal stochastic newton algorithms},  (2025), \url{https://doi.org/10.48550/arXiv.2401.10923}.

\bibitem{geneMatrixComps}
{\sc G.~H. Golub and C.~F. Van~Loan}, {\em Matrix Computations}, The Johns Hopkins University Press, Baltimore, 3rd~ed., 2013.

\bibitem{medical_imaging}
{\sc G.~T. Gullberg, R.~H. Huesman, B.~W. Reutter, J.~Qi, and D.~N.~G. Roy}, {\em Estimation of the parameter covariance matrix for a one-compartment cardiac perfusion model estimated from a dynamic sequence reconstructed using map iterative reconstruction algorithms},  (2004), \url{https://doi.org/10.2172/928329}.

\bibitem{Horn_Johnson_2012}
{\sc R.~A. Horn and C.~R. Johnson}, {\em Matrix Analysis}, Cambridge University Press, Cambridge; New York, 2nd~ed., 2012.

\bibitem{Householder_1965}
{\sc A.~Householder}, {\em The Theory of Matrices in Numerical Analysis}, Blaisdell Publishing Company, 1965.

\bibitem{Hutchinson}
{\sc M.~Hutchinson}, {\em A stochastic estimator of the trace of the influence matrix for {L}aplacian smoothing splines}, Communications in Statistics - Simulation and Computation, 19 (1990), pp.~433--450, \url{https://doi.org/10.1080/03610919008812866}.

\bibitem{TAKAHASHI}
{\sc K.Takahashi, J.~Fagan, and M.-S. Chin}, {\em Formation of sparse bus impedance matrix and its application to short circuit study}, Proc. PICA Conference, June, 1973,  (1973).

\bibitem{FIND.ALG}
{\sc S.~Li, S.~Ahmed, G.~Klimeck, and E.~Darve}, {\em Computing entries of the inverse of a sparse matrix using the {FIND} algorithm}, Journal of Computational Physics, 227 (2008), pp.~9408--9427, \url{https://doi.org/10.1016/j.jcp.2008.06.033}.

\bibitem{Pan_newton_inverse}
{\sc V.~Pan and J.~Reif}, {\em Fast and efficient parallel solution of dense linear systems}, Computers \& Mathematics with Applications, 17 (1989), p.~1481–1491, \url{https://doi.org/https://doi.org/10.1016/0898-1221(89)90081-3}.

\bibitem{pap_yuille_2010}
{\sc G.~Papandreou and A.~L. Yuille}, {\em Gaussian sampling by local perturbations}, Advances in Neural Information Processing Systems, 23 (2010), p.~1858–1866.

\bibitem{papandyuille}
{\sc G.~Papandreou and A.~L. Yuille}, {\em Efficient variational inference in large-scale {B}ayesian compressed sensing}, in 2011 IEEE International Conference on Computer Vision Workshops (ICCV Workshops), 2011, pp.~1332--1339, \url{https://doi.org/10.1109/ICCVW.2011.6130406}.

\bibitem{SPD_inversion}
{\sc E.~S. Quintana, G.~Quintana, X.~Sun, and R.~van~de Geijn}, {\em A note on parallel matrix inversion}, SIAM Journal on Scientific Computing, 22 (2001), pp.~1762--1771, \url{https://doi.org/10.1137/S1064827598345679}.

\bibitem{RUE20073177}
{\sc H.~Rue and S.~Martino}, {\em Approximate {B}ayesian inference for hierarchical {G}aussian {M}arkov random field models}, Journal of Statistical Planning and Inference, 137 (2007), pp.~3177--3192, \url{https://doi.org/10.1016/j.jspi.2006.07.016}.

\bibitem{Finn}
{\sc P.~Sidén, F.~Lindgren, D.~Bolin, and M.~Villani}, {\em Efficient covariance approximations for large sparse precision matrices}, Journal of Computational and Graphical Statistics, 27 (2018), pp.~898--909, \url{https://doi.org/10.1080/10618600.2018.1473782}.

\bibitem{SchurComplementTxbk}
{\sc F.~Zhang}, {\em The Schur complement and its applications}, Numerical methods and algorithms ; v. 4, Springer, New York, 1st~ed., 2005.

\bibitem{Rue_2023}
{\sc A.~Zhumekenov, E.~T. Krainski, and H.~Rue}, {\em Parallel selected inversion for space-time {G}aussian {M}arkov random fields},  (2023), \url{https://doi.org/10.48550/arXiv.2309.05435}.

\end{thebibliography}
\end{document}


\maketitle

\section{A detailed example}

Here we include some equations and theorem-like environments to show
how these are labelled in a supplement and can be referenced from the
main text.
Consider the following equation:
\begin{equation}
  \label{eq:suppa}
  a^2 + b^2 = c^2.
\end{equation}
You can also reference equations such as \cref{eq:matrices,eq:bb} 
from the main article in this supplement.

\lipsum[100-101]

\begin{theorem}
  An example theorem.
\end{theorem}

\lipsum[102]
 
\begin{lemma}
  An example lemma.
\end{lemma}

\lipsum[103-105]

Here is an example citation: \cite{KoMa14}.

\section[Proof of Thm]{Proof of \cref{thm:bigthm}}
\label{sec:proof}
\lipsum[106-112]

\section{Additional experimental results}
\Cref{tab:foo} shows additional
supporting evidence. 

\begin{table}[htbp]
{\footnotesize
  \caption{Example table}  \label{tab:foo}
\begin{center}
  \begin{tabular}{|c|c|c|} \hline
   Species & \bf Mean & \bf Std.~Dev. \\ \hline
    1 & 3.4 & 1.2 \\
    2 & 5.4 & 0.6 \\ \hline
  \end{tabular}
\end{center}
}
\end{table}

\bibliographystyle{siamplain}
\bibliography{references}